\documentclass[11pt]{article}

\usepackage[a4paper,margin=2.5cm]{geometry}
\usepackage{amsmath,amssymb,amsthm,mathtools}
\usepackage{booktabs}
\usepackage{enumitem}
\usepackage{hyperref}
\usepackage{url}
\usepackage[expansion=false]{microtype}

\numberwithin{equation}{section}
\theoremstyle{plain}
\newtheorem{theorem}{Theorem}[section]
\newtheorem{lemma}[theorem]{Lemma}
\newtheorem{proposition}[theorem]{Proposition}
\newtheorem{corollary}[theorem]{Corollary}

\theoremstyle{definition}
\newtheorem{definition}[theorem]{Definition}
\newtheorem{example}[theorem]{Example}

\theoremstyle{remark}
\newtheorem{remark}[theorem]{Remark}
\newtheorem{openproblem}[theorem]{Open Problem}

\newcommand{\N}{\mathbb{N}}
\newcommand{\Z}{\mathbb{Z}}

\begin{document}

\title{The Golden Sieve}

\author{Benoit Cloitre\\
\small Paris, France}

\date{March 2026}

\maketitle

\vspace{0.5em}
{\centering\itshape To Samuel Beatty, on the centenary of his problem.\par}
\vspace{0.5em}

\begin{abstract}
We revisit the golden sieve, a self-referential deletion process on
increasing sequences of positive integers introduced by the author
in 2002 (OEIS \texttt{A099267},~\cite{OEIS}).  Applied to the natural numbers, the sieve
produces the Wythoff pair as a Beatty partition.  For arithmetic progressions $a\N+b$,
we establish a connection with the $(j,x,y,z)$-hiccup sequences
recently studied by Fokkink and Joshi~\cite{FJ2026} and with
Fraenkel's complementary partitions.  We further introduce an
extraction sieve that also produces hiccup sequences, and whose
action on arithmetic progressions is governed by an explicit affine
transformation of hiccup parameters.
\end{abstract}



\section{Introduction}\label{sec:intro}

Let $W_0=(w_1<w_2<\cdots)$ be an infinite strictly increasing sequence of
positive integers.  The \emph{golden sieve} on~$W_0$ is the following
deterministic deletion process.  At step~$n$, one reads the $n$-th entry
$h_n$ of the current working sequence~$W_{n-1}$.  This value serves as a
\emph{position index}, and the element sitting at position~$h_n$
in~$W_{n-1}$ is removed.  Iterating produces a partition
\[
  W_0\;=\;\{s_n : n\ge1\}\;\sqcup\;\{d_n : n\ge1\},
\]
where $\sqcup$ denotes disjoint union, $(s_n)$ is the \emph{survivor}
sequence and $(d_n)$ the \emph{deletion} sequence, both strictly
increasing.

The mechanism is elementary, yet the output is rigid.  For the main
families $W_0=\N$ and $W_0=a\N+b$, the successive gaps of $(s_n)$ take exactly two values,
the hallmark of self-referential recurrences.
The name \emph{golden sieve} reflects the fact that for $W_0=\N$ the output
is controlled by the golden ratio $\varphi=(1+\sqrt{5})/2$: survivor slopes satisfy
the same family of quadratic equations $a\alpha^2-a\alpha-1=0$ that
has $\varphi$ as its $a=1$ root.

\smallskip
\noindent\textbf{The case $W_0=\N$.}
When the sieve starts from the natural numbers, the survivor sequence
(OEIS~A099267) coincides (up to a one-step shift) with the lower Wythoff sequence, the
gap word~$H$ is the Fibonacci word, and the partition $\N=\{s_n\}\sqcup\{d_n\}$
is a Beatty partition governed by the golden ratio.  Both $(s_n)$
and $(d_n)$ satisfy a self-referential two-gap recurrence of the form
\[
  x_n = x_{n-1}+g(n), \qquad g(n)\in\{g_0,g_1\},
\]
where the gap $g(n)$ depends on whether~$n$ belongs to the value set
of the sequence $(x_n)$ itself.

\smallskip
\noindent\textbf{Arithmetic progressions $W_0=a\N+b$.}
Our main results concern this family.  We establish:
\begin{enumerate}[label=(\roman*)]
\item a \emph{pointer--survivor identity} showing that the pointer
  always reads from the survivor prefix;
\item a \emph{two-gap property}, namely that the normalized survivor gaps lie
  in~$\{1,2\}$;
\item a \emph{rank identity} relating survivors and deletions by the
  affine relation $\delta_n = a\sigma_n + n + (b-1)$
  (where $\sigma_n$, $\delta_n$ are the normalized survivors and
  deletions, Definition~\ref{def:normalized});
\item a \emph{hiccup rule} describing both $(s_n)$ and $(d_n)$ by
  explicit self-referential recurrences (Theorem~\ref{thm:hiccup-ab}).
\end{enumerate}
The partition satisfies a Fraenkel-type complementary equation,
connecting the sieve to the game-theoretic constructions
of~\cite{Fraenkel1969,Fraenkel1998} and to Kimberling's complementary
equation framework~\cite{Kimberling1993,Kimberling2007,Kimberling2008}.
We also identify the sieve with Kimberling's rank transform
(Theorem~\ref{thm:ranktransform-new}).

\smallskip
\noindent\textbf{Related work.}
The golden sieve was introduced by the author in 2002 as OEIS entry A099267~\cite{OEIS}.
The broader study of self-referential two-gap recurrences grew out of~\cite{CSV2003},
which examined monotonic sequences satisfying $a(a(n))=un+v$; many instances were
contributed to the OEIS from 2002 onward.  The author later undertook a
systematic study~\cite{Cloitre2025}.  Independently, Fokkink and
Joshi~\cite{FJ2026} studied the same family and coined the term
\emph{hiccup sequences}, which we adopt throughout.
The golden sieve provides a dynamical origin for a distinguished
subfamily of hiccup sequences.

A prominent hiccup example is the sequence A086377 studied by Bosma,
Dekking, and Steiner~\cite{BDS2018}, later covered by the general
theory of~\cite{FJ2026}.  That sequence does \emph{not} arise from the
golden sieve, since its slope $1+\sqrt{2}$ satisfies $\alpha^2=2\alpha+1$,
not the characteristic family $a\alpha^2-a\alpha-1=0$ of the golden
sieve.  It is instead the output of a different self-referential
process, the \emph{extraction sieve} $\mathcal{C}_{1,3,2}$
(the ``silver sieve''),
introduced in Section~\ref{sec:metallic}.

\smallskip
\noindent\textbf{Outline.}
Section~\ref{sec:definition} defines the sieve, works out the case
$W_0=\N$, and derives the Wythoff--Beatty identification.
Section~\ref{sec:general} develops the theory for $W_0=a\N+b$: it
establishes the rank identity, the two-gap property, and the hiccup
rule showing that the survivor sequence is a hiccup sequence in the
sense of Fokkink--Joshi~\cite{FJ2026}, with the deletion gaps governed
by the same binary control word through a filtered self-referential law.
Section~\ref{sec:fraenkel} connects the sieve to Fraenkel's
complementary equations and Kimberling's rank transform.
Section~\ref{sec:variants} studies the sieve on perfect squares.
Section~\ref{sec:metallic} introduces the extraction sieve.
Section~\ref{sec:concl} collects open questions.

\section{The golden sieve on \texorpdfstring{$\N$}{N}}\label{sec:definition}

\subsection{Notation}
Throughout this paper, $\N=\{1,2,3,\ldots\}$ denotes the set of
positive integers and $\N_0=\{0,1,2,\ldots\}$ the set of nonnegative
integers.  We write $\mathbf{1}_E$ for the indicator function of a
set or event~$E$.

For an increasing sequence $X=(x_1<x_2<\cdots)$ of positive integers,
we write $X(k)=x_k$ for its $k$th term.
If $p\ge1$, we denote by $X\setminus\{X(p)\}$ the increasing sequence
obtained by deleting the \emph{value} $X(p)$ (equivalently, deleting the
entry at \emph{position}~$p$).

\smallskip\noindent\textbf{Notational convention.}
Three related symbols appear throughout: the pointer sequence~$(h_n)$
(lowercase italic), the binary gap word~$H$ (uppercase italic,
Definition~\ref{def:gapword-H}), and Kimberling's counting
function~$\tilde h$ (lowercase with tilde,
Definition~\ref{def:ranktransform-new}).  The three are typographically
distinct and play different roles.

\subsection{Definition}

\begin{definition}[Golden sieve]\label{def:sieve}
Let $W_0 = (w_1, w_2, w_3, \ldots)$ be a strictly increasing sequence
of positive integers.  Define inductively:
\begin{itemize}
\item The \emph{pointer} at step~$n$:
  $h_n := W_{n-1}(n)$ (the $n$-th element of the current working
  sequence~$W_{n-1}$).
\item The \emph{deletion}: $d_n := W_{n-1}(h_n)$ (the element at
  position~$h_n$ in~$W_{n-1}$).
\item The \emph{updated sequence}: $W_n := W_{n-1} \setminus \{d_n\}$.
\end{itemize}
The \emph{survivor sequence} is the set of elements never deleted:
$\{s_n : n \ge 1\} := W_0 \setminus \{d_1, d_2, \ldots\}$,
listed in increasing order.
\end{definition}

We first verify that the inductive process is well-defined.

\begin{lemma}\label{lem:well-defined}
For every $n\ge0$, the working sequence $W_n$ defined in
Definition~\ref{def:sieve} is an infinite increasing sequence of positive
integers. In particular, for each $n\ge1$ the pointer $h_n=W_{n-1}(n)$ and
the deleted value $d_n=W_{n-1}(h_n)$ are well-defined.
\end{lemma}

\begin{proof}
We argue by induction on~$n$. By assumption, $W_0$ is infinite increasing.
Assume $W_{n-1}$ is infinite increasing. Removing a single value from an
infinite increasing sequence yields again an infinite increasing sequence,
so $W_n=W_{n-1}\setminus\{d_n\}$ is infinite increasing.
Consequently, for $n\ge1$ the term $W_{n-1}(n)$ exists and defines the
pointer $h_n\in\N$. Since $W_{n-1}$ is infinite, its $h_n$-th term also
exists, hence $d_n=W_{n-1}(h_n)$ is well-defined.
\end{proof}

\begin{remark}%
\label{rem:degeneracy-initial}
The sieve uses a position--value feedback loop.  The pointer $h_n$ is a
\emph{value} read at position~$n$, and it is then used as a \emph{position}
to select the target $d_n=W_{n-1}(h_n)$. Lemma~\ref{lem:well-defined}
ensures that all pointers and targets exist for every step $n\ge1$.

If $w_1=1$ then $h_1=W_0(1)=1$ and the first deletion is degenerate, with
$d_1=W_0(h_1)=W_0(1)=1$ (pointer and target coincide). For the arithmetic
progressions $W_0=a\N+b$ with $a\ge2$, one has $w_1=a+b\ge3$ and this
degeneracy never occurs.
\end{remark}

\begin{lemma}\label{lem:rank-N}
Let $W_0=\N$, i.e.\ $(a,b)=(1,0)$.  Then $d_1=1$, and for every $n\ge2$
\begin{equation}\label{eq:rank-N}
  d_n = s_n + n - 1,
  \qquad\text{equivalently}\qquad
  \delta_n = \sigma_n + n - 1.
\end{equation}
\end{lemma}

\begin{proof}
The first step is degenerate: $h_1=W_0(1)=1$ and $d_1=W_0(1)=1$.

For $n\ge2$, after the deletion of $d_1=1$ the working sequence $W_1=(2,3,4,\ldots)$
has $w_1=2$, so the pointer $h_n=W_{n-1}(n)\ge n+1>n$ for all $n\ge2$.
Hence at step~$n$ the deletion target $d_n=W_{n-1}(h_n)$ lies strictly beyond
position~$n$ in $W_{n-1}$, which means the prefix of length $n-1$ of $W_{n-1}$
is unchanged at the next step.  By induction, $W_{n-1}(n)=s_n$ for all $n\ge2$.

Among the positive integers $1,2,\ldots,d_n$, exactly $n-1$ have been deleted
by step~$n$ (since $d_1,\ldots,d_{n-1}$ are all distinct and at most $d_n$).
The rank of $d_n$ in the surviving integers is therefore $d_n-(n-1)$.
But this rank also equals $s_n$ (the position at which $d_n$ was selected),
so $d_n-(n-1)=s_n$, giving~\eqref{eq:rank-N}.
\end{proof}

\subsection{Example: the sieve on \texorpdfstring{$\N$}{N}}

\begin{example}\label{ex:sieve-N}
Let $W_0=(1,2,3,4,5,6,7,8,\ldots)$. The first steps are:
\[
\begin{array}{c|c|c|l}
n & h_n=W_{n-1}(n) & d_n=W_{n-1}(h_n) & \text{prefix of }W_n \\ \hline
1 & 1 & 1 & (2,3,4,5,6,7,8,9,10,\ldots) \\
2 & 3 & 4 & (2,3,5,6,7,8,9,10,11,\ldots) \\
3 & 5 & 7 & (2,3,5,6,8,9,10,11,12,\ldots) \\
4 & 6 & 9 & (2,3,5,6,8,10,11,12,13,\ldots) \\
5 & 8 & 12 & (2,3,5,6,8,10,11,13,14,\ldots)
\end{array}
\]
Continuing indefinitely produces the deleted and survivor sequences
\begin{align*}
(d_n) &= 1,4,7,9,12,15,17,20,22,25,28,30,33,\ldots,\\
(s_n) &= 2,3,5,6,8,10,11,13,14,16,18,19,21,\ldots\,.
\end{align*}
Note that $d_1 = h_1 = 1$, so the pointer and target coincide at the
first step.  This degeneracy is specific to $W_0=\N$
(where $W_0(1)=1$) and does not occur when $w_1\ge2$.
\end{example}

\subsection{The hiccup rule}\label{sec:hiccup}

We begin with the basic instance $W_0=\N$, which already contains the
self-referential phenomenon that motivates this paper. Recall from
Example~\ref{ex:sieve-N} that the sieve produces two complementary
increasing sequences, namely the deleted values $(d_n)_{n\ge1}$ and the
survivors $(s_n)_{n\ge1}$, with
\[
\{d_1,d_2,\ldots\}\ \sqcup\ \{s_1,s_2,\ldots\}=\N.
\]
Empirically one observes that the gaps of both sequences take only two
values, and that the choice is governed by a membership test involving
the sequence itself.  This self-referential two-gap recurrence defines
the \emph{hiccup} paradigm (the term is due to~\cite{FJ2026}, while the
underlying constructions were introduced by the author in the OEIS from 2002 onward~\cite{OEIS}).

\begin{definition}[Hiccup sequences]\label{def:hiccup}
Let $y,z$ be distinct positive integers, let $x\ge0$, and let
$j\in\Z_{\ge0}$.
An increasing sequence $X=(x_n)_{n\ge1}$ of positive integers
with $x_1=x$ is called a \emph{$(j,x,y,z)$-hiccup sequence}
if for every $n\ge2$,
\begin{equation}\label{eq:hiccup-def}
  x_n-x_{n-1}=
  \begin{cases}
    y & \text{if } n-j\in\{x_1,x_2,\ldots,x_{n-1}\},\\
    z & \text{otherwise.}
  \end{cases}
\end{equation}
The gap at step~$n$ is thus determined by the membership of~$n-j$
in the value set of the sequence itself.
The parameter~$j$ controls the shift in the membership test:
$j=0$ probes~$n$, $j=1$ probes~$n-1$.
When $j=0$ and $x=1$, we recover the convention
of~\cite{Cloitre2025}.  The term ``hiccup'' and the
$(j,x,y,z)$-notation are due to Fokkink--Joshi~\cite{FJ2026}.
\end{definition}

\begin{example}\label{ex:hiccup-simple}
The $(1,1,2,1)$-hiccup sequence (i.e.\ $j=1$, $x_1=1$, large gap $y=2$ on
membership, small gap $z=1$ otherwise) begins
$1,3,4,6,8,9,11,12,14,16,17,\ldots$.
At step $n=2$: probe $n-1=1$; since $1\in\{1\}$, gap is $2$, giving $x_2=3$.
At step $n=3$: probe $n-1=2$; since $2\notin\{1,3\}$, gap is $1$, giving $x_3=4$.
This is the lower Wythoff sequence $(\lfloor n\varphi\rfloor)_{n\ge1}$, OEIS~A000201.
\end{example}

The golden sieve on~$\N$ produces sequences of each type.

\begin{theorem}\label{thm:hiccup}
Let $W_0=\N$ and let $(d_n)$ and $(s_n)$ be the deleted and survivor
sequences produced by the golden sieve.
\begin{enumerate}
\item[\textup{(i)}] \textbf{Deletions.}
The deletion sequence $(d_n)$ is the $(0,1,2,3)$-hiccup
sequence with $d_1=1$ \textup{(}OEIS~A007066\textup{)}.  Explicitly,
for every $n\ge2$,
\begin{equation}\label{eq:hiccup-N}
  d_n-d_{n-1}=
  \begin{cases}
    2 & \text{if } n\in\{d_1,d_2,\ldots\},\\
    3 & \text{if } n\notin\{d_1,d_2,\ldots\}.
  \end{cases}
\end{equation}
\item[\textup{(ii)}] \textbf{Survivors.}
The survivor sequence $(s_n)$ \textup{(}OEIS~A099267\textup{)} is,
for $n\ge3$, a $(0,\,s_1,\,2,\,1)$-hiccup sequence.  Explicitly,
\begin{equation}\label{eq:hiccup-N-surv}
  s_n - s_{n-1} =
  \begin{cases}
    2 & \text{if } n\in\{s_1,s_2,\ldots\},\\
    1 & \text{if } n\notin\{s_1,s_2,\ldots\}.
  \end{cases}
\end{equation}
\textup{(}The initial gap $s_2-s_1=1$ escapes the rule because
$2\in\{s_k\}$ but the degenerate first deletion $d_1=1$ shifts the
survivor sequence by one.\textup{)}
\end{enumerate}
The gap between consecutive deletions is $2$ or $3$, the gap between
consecutive survivors is $1$ or $2$, and in both cases the selection
is governed by the value set of the sequence itself.
\end{theorem}

\begin{proof}
A self-contained proof is given later in Corollary~\ref{cor:hiccup-N},
once the general analysis of the sieve on $a\N+b$ has been established.
We record the result here as the guiding example of the paper; the initial
degeneracy at $n=2$ (due to $d_1=h_1=1$) is explained in
Remark~\ref{rem:degeneracy-initial}.
\end{proof}

\begin{remark}\label{rem:hiccup-binary}
Setting $\varepsilon_n:=\mathbf{1}_{n\notin\{d_k\}}\in\{0,1\}$,
equation~\eqref{eq:hiccup-N} reads $d_n-d_{n-1}=2+\varepsilon_n$:
the gap word is a binary encoding of the deletion set via membership bits.
\end{remark}

\subsection{Wythoff identification}\label{sec:wythoff}

For $W_0 = \N$, the sieve partition is governed by the golden ratio
$\varphi = (1+\sqrt5)/2$.

\begin{corollary}\label{cor:wythoff}
For $W_0=\N$, the survivor and deleted sequences satisfy:
\begin{align}
  s_n &= \lfloor (n-1)\varphi \rfloor + 2
       \label{eq:surv-beatty}\\
  d_n &= \lfloor (n-1)\varphi^2 \rfloor + 2
       \qquad (n\ge2).
       \label{eq:del-beatty}
\end{align}
(The degenerate first deletion is $d_1=1$.)
Setting $A_m:=s_{m+1}-2$ and $B_m:=d_{m+1}-2$ for $m\ge1$ gives the
classical Wythoff pair
$A_m = \lfloor m\varphi\rfloor$ \textup{(A000201)},\;
$B_m = \lfloor m\varphi^2\rfloor$ \textup{(A001950)},\;
$B_m = A_m + m$.
The deleted sequence $(d_n)_{n\ge1}$ coincides with A007066.
\end{corollary}

\begin{proof}
A proof is given later, after the rank identity
(Proposition~\ref{thm:rank-ab}) and the mex-uniqueness lemma
(Lemma~\ref{lem:mex-uniqueness}) have been established.
\end{proof}

\begin{remark}\label{rem:beatty-intro}
The Beatty representation is specific to $a=1$.  For $a \ge 2$,
the survivor and deletion sequences are \emph{not} Beatty sequences
(Corollary~\ref{cor:beatty-dichotomy} in~\S\ref{sec:fraenkel}).
\end{remark}

The Beatty formula has a striking consequence for the Fibonacci numbers.

\begin{corollary}\label{cor:fib-invariance}
Let $(F_k)_{k\ge0}$ denote the Fibonacci sequence ($F_0=0$, $F_1=1$,
$F_{k+1}=F_k+F_{k-1}$).  Then for every $k\ge2$,
\begin{equation}\label{eq:fib-invariance}
  s(F_k) = F_{k+1}.
\end{equation}
In particular, the Fibonacci numbers form an invariant subset of
the survivor function, and $s^m(F_k) = F_{k+m}$ for all $m\ge0$.
\end{corollary}

\begin{proof}
By Corollary~\ref{cor:wythoff}, $s_n = \lfloor(n-1)\varphi\rfloor+2$
for $n\ge1$.  It therefore suffices to show that
$\lfloor(F_k-1)\varphi\rfloor = F_{k+1}-2$ for every $k\ge2$.
By the Binet formula,
\[
  F_k\,\varphi = F_{k+1} + \frac{(-1)^{k+1}}{\varphi^k}.
\]
Hence $(F_k-1)\varphi = F_{k+1}-\varphi + (-1)^{k+1}/\varphi^k$, and
\[
  \lfloor(F_k-1)\varphi\rfloor
  = F_{k+1} - 2 + \bigl\lfloor 2-\varphi+(-1)^{k+1}/\varphi^k\bigr\rfloor.
\]
Since $2-\varphi = 1/\varphi^2 \approx 0.382$
and $|1/\varphi^k|\le 1/\varphi^2$ for $k\ge2$, the
quantity $2-\varphi+(-1)^{k+1}/\varphi^k$ lies in $[0,1)$ for all
$k\ge2$ (with equality at the left endpoint only when $k=2$).
Its floor is therefore~$0$, giving
$\lfloor(F_k-1)\varphi\rfloor = F_{k+1}-2$ as claimed.
The iterated identity $s^m(F_k)=F_{k+m}$ follows by induction.
\end{proof}

\section{The golden sieve on arithmetic progressions}\label{sec:general}

\subsection{Setting and normalization}

We now develop the theory for the golden sieve on
$W_0 = a\N+b := \{an+b : n \ge 1\}$, where $a \ge 1$ and $0\le b < a$
(with the convention $a\N+0 = \{a, 2a, 3a, \ldots\}$).

\begin{definition}[Normalized sequences]\label{def:normalized}
The \emph{normalization map} $w = an+b \mapsto n$ identifies the
$n$-th element of~$W_0$ with the integer~$n$.  We write:
\begin{itemize}
\item $\sigma_n := (s_n - b)/a$ for the normalized survivors,
\item $\delta_n := (d_n - b)/a$ for the normalized deletions.
\end{itemize}
Both $(\sigma_n)$ and $(\delta_n)$ are strictly increasing sequences
of positive integers, and
$\{\sigma_n : n \ge 1\} \sqcup \{\delta_n : n \ge 1\} = \N$.
\end{definition}

The normalization reduces the sieve on~$W_0$ to a sieve on~$\N$ with
a modified pointer, namely $h_n = s_n = a\sigma_n + b$ in the original model.
For $a\ge2$ one has $h_n\ge 2n>n$, so the pointer is always strictly ahead
of position~$n$; this rules out the degenerate first step and is the source
of the $a=1$ versus $a\ge2$ dichotomy throughout the paper.

\subsection{Reduction to the index model}\label{sec:twogap}
Let $W_0=W_0(a,b)=\{an+b:n\ge1\}$ and let $\nu(an+b)=n$ as in
Definition~\ref{def:normalized}.  At each step the working sequence
$W_{n-1}$ is obtained from $W_0$ by deleting the values
$d_1,\ldots,d_{n-1}$.  Applying $\nu$ entrywise, it is convenient to
work with the corresponding working \emph{index} sequence
\[
  I_{n-1}:=\nu(W_{n-1})\subseteq\N,
\]
so that $I_{n-1}$ is the increasing sequence obtained from $\N$ by
deleting $\delta_1,\ldots,\delta_{n-1}$.
Since $W_{n-1}(k)=a\,I_{n-1}(k)+b$, the pointer and deletion indices are
\begin{equation}\label{eq:index-model}
  \eta_n = I_{n-1}(n),\qquad
  \delta_n = I_{n-1}(h_n)\qquad (n\ge1),
\end{equation}
where $h_n=a\eta_n+b$ is the pointer value in the original scale.

\subsection{Pointer--survivor identity (prefix stabilization)}

The key structural fact is that the sieve prefix stabilizes early.

\begin{lemma}\label{lem:prefix-stability}
Let $V=(v_1<v_2<\cdots)$ be a strictly increasing sequence of positive
integers, and let $V^{(p)}$ be the sequence obtained by deleting the
entry at position~$p$.  If $p>n$, then
\[
  V^{(p)}(k)=v_k \qquad (1\le k\le n).
\]
\end{lemma}

\begin{proof}
Deleting the entry at position~$p$ leaves positions $1,\dots,p-1$
unchanged and shifts only the entries with index $>p$ one step to the
left.  Since $p>n$, the first $n$ entries are unchanged.
\end{proof}

\begin{proposition}\label{thm:pointer-survivor}
Assume $(a,b)\neq(1,0)$.
Then for every $n\ge1$ the first $n$ elements of $W_{n-1}$ are exactly
the first $n$ surviving values $s_1,\ldots,s_n$.  In particular,
\[
  h_n=s_n \quad\text{and}\quad \eta_n=\sigma_n \qquad (n\ge1).
\]
\end{proposition}

\begin{proof}
Fix $n\ge1$.  Since $(a,b)\neq(1,0)$, we have $w_m=am+b\ge m+1$ for all
$m\ge1$.  At step~$m$, the pointer satisfies
\[
  h_m = W_{m-1}(m) \ge W_0(m) = w_m \ge m+1,
\]
because deleting entries from an increasing sequence cannot decrease its
$m$-th term.  Hence step~$m$ deletes an entry at position $h_m\ge m+1>m$.

Now let $m\ge n$.  Then $h_m\ge m+1>n$, so by
Lemma~\ref{lem:prefix-stability}, step~$m$ does not change the first~$n$
entries of the working sequence.  Therefore the prefix of length~$n$
stabilizes at time~$n-1$: the tuple
$(W_{n-1}(1),\dots,W_{n-1}(n))$ is never modified afterwards.

Since these $n$ entries are never deleted, they are survivors.
Conversely, any survivor $t<W_{n-1}(n)$ must already appear among the
first $n$ entries of $W_{n-1}$, since $W_{n-1}$ is increasing and
contains all undeleted elements up to that stage.  It follows that
\[
  (W_{n-1}(1),\dots,W_{n-1}(n)) = (s_1,\dots,s_n),
\]
and in particular $h_n=W_{n-1}(n)=s_n$.
Applying $\nu$ gives $\eta_n=\nu(h_n)=\nu(s_n)=\sigma_n$.
\end{proof}

\begin{lemma}\label{lem:mono-deletions}
Assume $(a,b)\neq(1,0)$. Then $(d_n)$ and $(\delta_n)$ are strictly
increasing.
\end{lemma}

\begin{proof}
By Proposition~\ref{thm:pointer-survivor}, $(h_n)=(s_n)$ is strictly
increasing, hence $h_{n+1}>h_n$.
At step $n$ the value $d_n$ is deleted from position $h_n$ of $W_{n-1}$.
Since $h_{n+1}>h_n$, the element that will be read at position $h_{n+1}$
in $W_n$ was at position $h_{n+1}+1$ in $W_{n-1}$. Therefore
\[
  d_{n+1}=W_n(h_{n+1})=W_{n-1}(h_{n+1}+1) > W_{n-1}(h_n)=d_n.
\]
Applying $\nu$ yields $\delta_{n+1}>\delta_n$.
\end{proof}

\subsection{Rank identity}

We now derive the affine relation between survivors and deletions.

\begin{proposition}\label{thm:rank-ab}
Assume $(a,b)\neq(1,0)$. For every $n\ge1$,
\begin{equation}\label{eq:rank-ab}
  \delta_n = a\sigma_n + n + (b-1).
\end{equation}
Equivalently, in the original scale,
$d_n = a\delta_n+b = a\bigl(a\sigma_n + n + (b-1)\bigr)+b$.
\end{proposition}

\begin{proof}
Fix $n\ge1$. By Proposition~\ref{thm:pointer-survivor} we have
$h_n=a\sigma_n+b$. In the index model, $I_{n-1}$ is obtained from $\N$ by
deleting the $n-1$ indices $\delta_1<\cdots<\delta_{n-1}$
(Lemma~\ref{lem:mono-deletions}).
The deleted index $\delta_n$ is the element of $I_{n-1}$ at \emph{position}
$h_n$ (see \eqref{eq:index-model}), so the rank of $\delta_n$ inside
$I_{n-1}$ is $h_n$.

Among the integers $\le \delta_n$, exactly $n-1$ have been deleted, hence
the rank of $\delta_n$ in $I_{n-1}$ is $\delta_n-(n-1)$. Therefore
\[
  \delta_n-(n-1)=h_n=a\sigma_n+b,
\]
which rearranges to \eqref{eq:rank-ab}.
\end{proof}

\subsection{Two-gap property and gap dichotomy for deletions}

\begin{proposition}\label{thm:two-gap}
Assume $a\ge2$. Then the normalized survivor gaps satisfy
\[
  \sigma_n-\sigma_{n-1}\in\{1,2\}\qquad (n\ge2).
\]
\end{proposition}

\begin{proof}
Subtract \eqref{eq:rank-ab} at $n$ and $n-1$:
\[
  \delta_n-\delta_{n-1}
  = a(\sigma_n-\sigma_{n-1})+1.
\]
Since $\sigma_n-\sigma_{n-1}\ge1$ we get $\delta_n-\delta_{n-1}\ge a+1\ge3$.
Thus no two deleted indices are consecutive.
Because $(\sigma_n)$ is the increasing complement of $(\delta_n)$ in $\N$,
between two successive survivors there can be at most one missing integer,
so $\sigma_n-\sigma_{n-1}\le2$.
\end{proof}

\begin{definition}[Binary gap word]\label{def:gapword-H}
For $n\ge2$ define the binary word
\begin{equation}\label{eq:H-def}
  H(n):=(\sigma_n-\sigma_{n-1})-1\in\{0,1\}.
\end{equation}
Thus $H(n)=0$ corresponds to a survivor gap~$1$, and $H(n)=1$ to a gap~$2$.
\end{definition}

The same word controls the deletion gaps.

\begin{proposition}\label{prop:del-gaps-H}
Assume $(a,b)\neq(1,0)$. For every $n\ge2$,
\begin{align}
  \delta_n-\delta_{n-1} &= (a+1)+a\,H(n), \label{eq:delta-gaps-H}\\
  d_n-d_{n-1} &= a(a+1)+a^2\,H(n). \label{eq:d-gaps-H}
\end{align}
Equivalently, $\delta_n-\delta_{n-1}\in\{a+1,2a+1\}$ and
$d_n-d_{n-1}\in\{a(a+1),a(2a+1)\}$, with the larger gap occurring
exactly when $\sigma_n=\sigma_{n-1}+2$.
\end{proposition}

\begin{proof}
Equation \eqref{eq:delta-gaps-H} is the difference identity
$\delta_n-\delta_{n-1}=a(\sigma_n-\sigma_{n-1})+1$
rewritten using $\sigma_n-\sigma_{n-1}=1+H(n)$. Multiplying by $a$ gives
\eqref{eq:d-gaps-H}.
\end{proof}

\subsection{The binary gap word}\label{sec:gapword}

Recall from Definition~\ref{def:gapword-H} that $H(n):=(\sigma_n-\sigma_{n-1})-1\in\{0,1\}$
encodes the two-gap phenomenon, with $H(n)=0$ for a gap of~$1$ and $H(n)=1$ for a gap of~$2$.
By Proposition~\ref{prop:del-gaps-H}, the same word controls the deletion gaps:
\[
  \delta_n-\delta_{n-1}=(a+1)+a\,H(n)\qquad(n\ge2).
\]

\subsubsection{Reconstructing \texorpdfstring{$\sigma$}{sigma} from $H$}

For $a\ge2$, the first normalized survivor is $\sigma_1=1$
(since the first deletion index satisfies $\delta_1\ge2$).
The word $H$ then determines the full survivor sequence by telescoping.

\begin{lemma}\label{lem:reconstruct-sigma}
Assume $a\ge2$. Then for every $n\ge1$,
\begin{equation}\label{eq:sigma-from-H}
  \sigma_n = n + \sum_{k=2}^{n} H(k),
\end{equation}
with the convention that the empty sum is $0$ for $n=1$.
\end{lemma}

\begin{proof}
For $n\ge2$ we have $\sigma_n-\sigma_{n-1}=1+H(n)$. Summing from $2$
to $n$ and using $\sigma_1=1$ gives
\[
  \sigma_n = 1 + \sum_{k=2}^{n}(1+H(k)) = n + \sum_{k=2}^{n}H(k).
\qedhere
\]
\end{proof}

\subsubsection{The case \texorpdfstring{$a=1$}{a=1}: Fibonacci/Sturmian word}

For $W_0=\N$ (equivalently $(a,b)=(1,0)$), Corollary~\ref{cor:wythoff}
gives the survivor sequence as
\[
  s_n=\lfloor (n-1)\varphi\rfloor+2,\qquad \varphi=\frac{1+\sqrt5}{2}.
\]
Define $H(n):=(s_n-s_{n-1})-1$ for $n\ge2$, so that $H$ is the binary
gap word for the sieve on~$\N$.

\begin{proposition}\label{prop:fib-word}
For $W_0=\N$, the binary gap word $H$ is a (characteristic) Sturmian word (see e.g.~\cite{Lothaire2002}).
More precisely, for $n\ge2$,
\begin{equation}\label{eq:H-fib-mech}
  H(n)=\Bigl\lfloor \frac{n-1}{\varphi}\Bigr\rfloor
      -\Bigl\lfloor \frac{n-2}{\varphi}\Bigr\rfloor,
\end{equation}
so $H$ is the mechanical word of slope $1/\varphi$ (up to the index shift
$n\mapsto n-1$). Equivalently, after exchanging the symbols $0$ and $1$,
one recovers the standard Fibonacci word.
\end{proposition}

\begin{proof}
Using $\varphi=1+1/\varphi$, we have
\[
  \lfloor (n-1)\varphi\rfloor
  =\Bigl\lfloor (n-1)+\frac{n-1}{\varphi}\Bigr\rfloor
  =(n-1)+\Bigl\lfloor \frac{n-1}{\varphi}\Bigr\rfloor,
\]
since $(n-1)/\varphi\notin\Z$. Therefore
\[
  s_n = \lfloor (n-1)\varphi\rfloor+2
      = n+1 + \Bigl\lfloor \frac{n-1}{\varphi}\Bigr\rfloor.
\]
Subtracting the same identity at $n-1$ yields
\[
  s_n-s_{n-1}
  = 1 + \Bigl(\Bigl\lfloor \frac{n-1}{\varphi}\Bigr\rfloor
               -\Bigl\lfloor \frac{n-2}{\varphi}\Bigr\rfloor\Bigr),
\]
which is exactly \eqref{eq:H-fib-mech}.
\end{proof}

\subsubsection{General \texorpdfstring{$a\ge2$}{a>=2}: density and examples}

Although $H$ is always binary, its combinatorial properties depend on
$(a,b)$. What is universal is the asymptotic frequency of $1$'s, which is
dictated by the characteristic slope $\alpha(a)$ of Theorem~\ref{thm:slope}.

\begin{proposition}\label{thm:density-H}
Assume $a\ge2$, and let $\alpha(a)\in(1,2)$ be the limit
$\alpha(a)=\lim_{n\to\infty}\sigma_n/n$ from Theorem~\ref{thm:slope}.
Then the frequency of $1$'s in $H$ exists and equals
\begin{equation}\label{eq:rho-a}
  \rho(a):=\lim_{n\to\infty}\frac1{n}\sum_{k=2}^{n}H(k)=\alpha(a)-1
  =\frac{\sqrt{a^2+4a}-a}{2a}.
\end{equation}
\end{proposition}

\begin{proof}
By Lemma~\ref{lem:reconstruct-sigma},
$\sum_{k=2}^{n}H(k)=\sigma_n-n$. Dividing by $n$ and letting $n\to\infty$
gives
\[
  \lim_{n\to\infty}\frac1n\sum_{k=2}^{n}H(k)
  =\lim_{n\to\infty}\Bigl(\frac{\sigma_n}{n}-1\Bigr)=\alpha(a)-1.
\]
Finally, substituting $\alpha(a)=\frac{a+\sqrt{a^2+4a}}{2a}$
(Theorem~\ref{thm:slope}) yields \eqref{eq:rho-a}.
\end{proof}

\begin{remark}
For $a=1$, Proposition~\ref{prop:fib-word} shows that $H$ is Sturmian.
For $a\ge2$, computations indicate that $H$ is typically \emph{not}
Sturmian. For instance, for $(a,b)=(2,0)$ one finds five distinct factors
of length~$3$ (hence factor complexity $p(3)=5>4$), which is impossible
for a Sturmian word.
\end{remark}

\begin{table}[t]
\centering
\begin{tabular}{@{}llp{0.78\textwidth}@{}}
\toprule
$(a,b)$ & word & prefix of $H(2),H(3),\ldots$ \\ \midrule
$(1,0)$ & (Sturmian) & $0101101011011010110101101101011011010110\ldots$ \\
$(2,0)$ & (binary)   & $1000101010100010100010100010100010101010\ldots$ \\
$(2,1)$ & (binary)   & $0101000101000101010100010100010101010001\ldots$ \\
$(3,0)$ & (binary)   & $0100100000100100100000100100100100100100\ldots$ \\
\bottomrule
\end{tabular}
\caption{Prefixes of the binary gap word $H$ for small parameters.}
\label{tab:gapword-prefix}
\end{table}

\begin{openproblem}\label{op:gapword}
For the golden sieve on $W_0=a\N+b$ with $a\ge2$, determine the
combinatorial class of the gap word~$H$.  Is it substitutive (fixed
point of a primitive morphism)?  Is it $k$-automatic for some~$k$, or
Ostrowski-automatic with respect to the continued fraction
of~$\alpha(a)$ (see~\cite{BSS2021})?  In particular, determine its factor complexity
$p_H(n)$ and its balance properties as functions of $(a,b)$.  (For
$a=1$, the gap word is the Fibonacci word, which is
Fibonacci-automatic but not $k$-automatic.
The tool \textsc{Walnut}~\cite{Shallit2022walnut}, which handles both
$k$-automatic and Ostrowski-automatic sequences, may be well suited
to some of these questions.)
\end{openproblem}

\subsection{Hiccup rule for \texorpdfstring{$a\N+b$}{aN+b}}\label{sec:hiccup-ab}

The same binary word $H$ controls both gap sequences.
For the survivors it gives a hiccup law.
For the deletions it gives a two-gap rule in which the membership test
is taken inside~$W_0$.

\begin{theorem}\label{thm:hiccup-ab}
Let $a\ge2$ and $W_0=a\N+b$.
The survivor sequence is a hiccup sequence in the sense of
Definition~\ref{def:hiccup}, and the deletion sequence satisfies a
filtered self-referential two-gap law.
\begin{enumerate}
\item[\textup{(i)}] \textbf{Survivors.}
The survivor sequence $(s_n)$ is a $(0,\,s_1,\,2a,\,a)$-hiccup
sequence in the sense of Definition~\textup{\ref{def:hiccup}}.
Explicitly, for every $n\ge2$,
\begin{equation}\label{eq:hiccup-survivor-rule}
  s_n - s_{n-1} =
  \begin{cases}
    2a & \text{if } n\in\{s_1,s_2,\ldots\},\\
    a  & \text{if } n\notin\{s_1,s_2,\ldots\}.
  \end{cases}
\end{equation}

\item[\textup{(ii)}] \textbf{Deletions.}
The deletion gaps take exactly two values $\{a(a\!+\!1),\;a(2a\!+\!1)\}$,
controlled by a self-referential test filtered through~$W_0$.
For every $n\ge2$,
\begin{equation}\label{eq:hiccup-deletion-rule}
  d_n - d_{n-1} =
  \begin{cases}
    a(2a+1) & \text{if } n\in W_0\setminus\{d_1,d_2,\ldots\},\\
    a(a+1)  & \text{otherwise.}
  \end{cases}
\end{equation}
Since $W_0 = \{s_k\}\sqcup\{d_k\}$, the test
``$n\in W_0\setminus\{d_k\}$'' is equivalent to
``$n\in\{s_1,s_2,\ldots\}$''.  When $a=1$ and $W_0=\N$, the
filter vanishes and \eqref{eq:hiccup-deletion-rule} reduces to the
standard hiccup rule $d_n-d_{n-1}=2$ if $n\in \{d_k\}$, $=3$ if
$n\notin \{d_k\}$ \textup{(}Theorem~\textup{\ref{thm:hiccup})}.
\end{enumerate}
Both gap sequences are encoded by a single binary word $H(n)\in\{0,1\}$
defined by
\[
  H(n) := \mathbf{1}_{n\in\{s_1,s_2,\ldots\}},
\]
so that $s_n-s_{n-1}=a(1+H(n))$ and
$d_n-d_{n-1}=a\bigl(a(1+H(n))+1\bigr)$.
\end{theorem}

\begin{proof}
By the two-gap property (Proposition~\ref{thm:two-gap}), the normalized
survivor gaps $\sigma_n-\sigma_{n-1}$ lie in~$\{1,2\}$.  A gap of~$2$
means that exactly one integer in~$(\sigma_{n-1},\sigma_n)$ is missing
from the survivor set, hence belongs to the deletion set~$\{\delta_k\}$.
Since $s_n=a\sigma_n+b$, a normalized gap of~$2$ translates to
$s_n-s_{n-1}=2a$, and the missing integer is $\sigma_{n-1}+1$, which
corresponds to the value $s_{n-1}+a$ in the original scale.

We claim that $s_n-s_{n-1}=2a$ if and only if $n\in\{s_1,s_2,\ldots\}$.
Since $\sigma_n-\sigma_{n-1}\in\{1,2\}$, a gap of~$2$ occurs exactly when
$\sigma_{n-1}+1\in\{\delta_k\}$.  In that case, write
$\sigma_{n-1}+1=\delta_j$.  The deletion counting function
$D(\sigma_{n-1}+1)$ counts $j$ deletions among $\{1,\ldots,\sigma_{n-1}+1\}$,
and $S(\sigma_{n-1})=n-1$ (since $\sigma_{n-1}$ is the $(n\!-\!1)$-th survivor).
Hence $j=(\sigma_{n-1}+1)-(n-1)=\sigma_{n-1}-n+2$.
Now the rank identity (Proposition~\ref{thm:rank-ab}) gives
$\delta_j=a\sigma_j+j+(b-1)$, so
\[
  \sigma_{n-1}+1 = a\sigma_j + (\sigma_{n-1}-n+2) + (b-1),
\]
which simplifies to $n = a\sigma_j + b = s_j$.  Thus
$\sigma_n-\sigma_{n-1}=2$ if and only if $n=s_j$ for some~$j$, i.e.\
$n\in\{s_1,s_2,\ldots\}$.
Since $s_n=a\sigma_n+b$, the gap in the original scale is
$s_n-s_{n-1}=2a$, giving~\eqref{eq:hiccup-survivor-rule}.

For the deletion gaps, $d_n = a\delta_n+b$ and
$\delta_n = a\sigma_n + n + (b-1)$ (Proposition~\ref{thm:rank-ab}),
whence
$d_n - d_{n-1} = a\bigl(a(\sigma_n-\sigma_{n-1})+1\bigr)
               = a\bigl(a(1+H(n))+1\bigr)$,
yielding~\eqref{eq:hiccup-deletion-rule}.
Since $W_0 = \{s_k\}\sqcup\{d_k\}$, the condition
$n\in W_0\setminus\{d_k\}$ is the same as $n\in\{s_k\}$,
i.e.\ $H(n)=1$.
\end{proof}

\begin{remark}\label{rem:deletions-not-hiccup}
For $a=1$ and $W_0=\N$, the membership test in
\eqref{eq:hiccup-deletion-rule} reduces to $n\in\{d_k\}$, which is the
standard self-referential test of Definition~\ref{def:hiccup}; the deletion
sequence is then itself a hiccup sequence.
For $a\ge2$, however, the test is performed in $W_0\setminus\{d_k\}$, not in
$\{d_k\}$ itself, so Theorem~\ref{thm:hiccup-ab} does \emph{not} assert that
$(d_n)$ is a hiccup sequence in the sense of
Definition~\ref{def:hiccup}.
Rather, the deletion gaps are governed by the same binary word $H$ as the
survivor gaps, but through a filtered ambient test rather than through
membership in the value set of $(d_n)$ itself.
\end{remark}

\begin{corollary}
\label{cor:hiccup-N}
The proof of Theorem~\ref{thm:hiccup-ab} extends, with the obvious
degeneracy at the first step, to the case $(a,b)=(1,0)$.
In this case the normalized model coincides with the original
($\sigma_n=s_n$, $\delta_n=d_n$), the survivor rule
\eqref{eq:hiccup-N-surv} holds for $n\ge3$, and the deletion sequence
satisfies the standard hiccup rule~\eqref{eq:hiccup-N}.
\end{corollary}

\begin{proof}
For $(a,b)=(1,0)$, the normalized model coincides with the original one:
\[
  \sigma_n=s_n,\qquad \delta_n=d_n.
\]
By Lemma~\ref{lem:rank-N} (which holds for $n\ge2$),
\[
  \delta_n=\sigma_n+n-1,
\]
hence
\[
  \delta_n-\delta_{n-1}=(\sigma_n-\sigma_{n-1})+1\in\{2,3\},
\]
since the survivor gaps satisfy $\sigma_n-\sigma_{n-1}\in\{1,2\}$.

We now determine when the larger gap $\delta_n-\delta_{n-1}=3$ occurs.
This is equivalent to $\sigma_n-\sigma_{n-1}=2$.  Since $\{s_k\}$ and $\{d_k\}$
partition~$\N$ (Lemma~\ref{lem:well-defined}), the complement of $\{d_k\}$ in $\N$
is exactly $\{s_k\}$.  The survivor sequence $(s_n)$ is the lower Wythoff sequence
(Corollary~\ref{cor:wythoff}), whose gap word is the Fibonacci word on $\{1,2\}$;
the gap $\sigma_n-\sigma_{n-1}=2$ occurs precisely when $n\in\{s_k\}$.
Therefore
\[
  d_n-d_{n-1}=3 \iff n\in\{s_k\},
  \qquad
  d_n-d_{n-1}=2 \iff n\in\{d_k\},
\]
which is exactly the standard hiccup rule~\eqref{eq:hiccup-N}.
The survivor rule~\eqref{eq:hiccup-N-surv} follows by the same reasoning:
$\sigma_n-\sigma_{n-1}=2$ iff $n\in\{s_k\}$, and $\sigma_n-\sigma_{n-1}=1$
iff $n\in\{d_k\}$, for $n\ge3$.
The exception at $n=2$ is due to the initial degeneracy $d_1=1=h_1$.
\end{proof}

\section{Fraenkel connection, games, and rank transforms}\label{sec:fraenkel}

The rank identity proved in Section~\ref{sec:general} places the golden sieve
for arithmetic progressions into the classical framework of
\emph{complementary equations} of Fraenkel
and Kimberling.

\subsection{Fraenkel-type complementary equation}

Recall that $(\sigma_n)$ and $(\delta_n)$ are the normalized survivor and
deletion sequences (Definition~\ref{def:normalized}), so that
$\{\sigma_n\}\sqcup\{\delta_n\}=\N$.
The general rank identity (Proposition~\ref{thm:rank-ab}) can be rewritten as the
following complementary equation, which will be the engine behind
the asymptotic and word-theoretic results.

\begin{corollary}\label{cor:fraenkel-oce}
Assume $a\ge2$ and $0\le b<a$. Then for every $n\ge1$,
\begin{equation}\label{eq:fraenkel-oce}
  \delta_n = a\,\sigma_n + n + (b-1).
\end{equation}
In particular, in the nondegenerate residue class $b=1$ we have the pure
Fraenkel form $\delta_n=a\sigma_n+n$.
\end{corollary}

\subsection{Game-theoretic interpretation}

Wythoff's game~\cite{Wythoff1907} is a two-player combinatorial game played
with two heaps of tokens. A position $(x,y)$ is a $\mathcal{P}$-position
(previous player wins) if and only if $\{x,y\}=\{A_n,B_n\}$ for some
$n\ge0$, where $A_n=\lfloor n\varphi\rfloor$ and
$B_n=\lfloor n\varphi^2\rfloor$ are the lower and upper Wythoff
sequences.  These satisfy $B_n=A_n+n$ and
$\{A_n\}\sqcup\{B_n\}=\N_0$.

For $W_0=\N$ (i.e.\ $(a,b)=(1,0)$), the golden sieve produces the
partition $\{\sigma_n\}\sqcup\{\delta_n\}=\N$.
By Corollary~\ref{cor:wythoff}, after the shift
\[
  A_m=\sigma_{m+1}-1, \qquad B_m=\delta_{m+1}-1 \qquad (m\ge1),
\]
one recovers the classical Wythoff pair, satisfying $B_m=A_m+m$.
The golden sieve thus gives a dynamical construction of the Wythoff
$\mathcal{P}$-positions through its self-referential deletion mechanism.

Fraenkel~\cite{Fraenkel1969,Fraenkel1998} generalized Wythoff's game to
a family of two-player combinatorial games $\mathrm{NIM}(a,1)$, indexed by $a\ge1$, whose
$\mathcal{P}$-positions satisfy $B_n = a A_n + n$ and form complementary
Beatty sequences when $a=1$, but admit no Beatty representation for
$a\ge2$.

\begin{lemma}\label{lem:mex-uniqueness}
Fix integers $a\ge1$ and $c\ge0$.  There is at most one pair of increasing
complementary sequences $(u_n)_{n\ge1}$, $(v_n)_{n\ge1}$ of positive integers
satisfying
\[
  v_n = a\,u_n + n + c \qquad (n\ge1).
\]
More precisely, the pair is recursively determined by $u_1=1$,
$v_1=a+1+c$, and for $n\ge2$,
\[
  u_n = \min\bigl(\N\setminus\{u_1,\dots,u_{n-1},v_1,\dots,v_{n-1}\}\bigr),
  \qquad
  v_n = a\,u_n+n+c.
\]
\end{lemma}

\begin{proof}
Since $a\ge1$ and $c\ge0$, we have $v_n - u_n = (a-1)u_n+n+c \ge n \ge 1$,
so $v_n > u_n$ for every $n$.  Hence the smallest positive integer not yet
appearing among $u_1,\dots,u_{n-1},v_1,\dots,v_{n-1}$ must be $u_n$ (since
$v_n$ lies strictly above it), and $v_n$ is then forced by the affine relation.
Uniqueness follows by induction.
\end{proof}

\begin{proof}[Proof of Corollary~\ref{cor:wythoff}]
For $(a,b)=(1,0)$, the rank identity~\eqref{eq:rank-N} holds for $n\ge2$
(Lemma~\ref{lem:rank-N}), giving $\delta_n = \sigma_n + n - 1$,
i.e.\ $d_n = s_n + n - 1$ for $n\ge2$.
Set $A_m := s_{m+1}-2$ and $B_m := d_{m+1}-2$ for $m\ge1$.  Then
\[
  B_m = d_{m+1}-2 = s_{m+1}+m-2 = A_m+m,
\]
and $(A_m),(B_m)$ form a complementary partition of~$\N$.
By Lemma~\ref{lem:mex-uniqueness} with $(a,c)=(1,0)$, there is at most one
such pair satisfying $B_m=A_m+m$.
The classical Wythoff pair $(\lfloor m\varphi\rfloor,\lfloor m\varphi^2\rfloor)$
is complementary and satisfies $\lfloor m\varphi^2\rfloor=\lfloor m\varphi\rfloor+m$
by the Rayleigh--Beatty theorem~\cite{Beatty1926,Rayleigh1894}.
Hence $A_m=\lfloor m\varphi\rfloor$ and $B_m=\lfloor m\varphi^2\rfloor$,
giving~\eqref{eq:surv-beatty}--\eqref{eq:del-beatty} after shifting indices.
\end{proof}

\begin{proposition}%
\label{prop:P-positions}
For every $a\ge1$, the golden sieve on $W_0=a\N+1$ produces, in
normalized indices, the $\mathcal{P}$-positions of $\mathrm{NIM}(a,1)$.
More precisely, $A_n=\sigma_n$ and $B_n=\delta_n$ for $n\ge1$, where
$(A_n,B_n)$ are the Fraenkel $\mathcal{P}$-position pairs and
$(\sigma_n,\delta_n)$ are the normalized survivor/deletion sequences.
\end{proposition}

\begin{proof}
For $b=1$, Corollary~\ref{cor:fraenkel-oce} gives $\delta_n=a\sigma_n+n$.
Thus $(\sigma_n),(\delta_n)$ form an increasing complementary pair satisfying
the same affine relation as Fraenkel's $\mathcal{P}$-positions $(A_n),(B_n)$
of $\mathrm{NIM}(a,1)$.  By Lemma~\ref{lem:mex-uniqueness} (with $c=0$),
such a pair is unique, so $A_n=\sigma_n$ and $B_n=\delta_n$.
\end{proof}

\subsection{A self-referential form for \texorpdfstring{$\sigma$}{sigma}}

For $a\ge2$, the deletions are never consecutive.  Indeed, since
$\sigma_{n}-\sigma_{n-1}\ge1$ and \eqref{eq:fraenkel-oce} holds,
\begin{equation}\label{eq:delta-gap-min}
  \delta_n-\delta_{n-1}=a(\sigma_n-\sigma_{n-1})+1 \ge a+1 \ge 3
  \qquad (n\ge2),
\end{equation}
so $\delta_n-1\notin\{\delta_1,\delta_2,\ldots\}$ for all $n\ge2$.
This lets us convert the two-sequence complementary equation into a one-sequence
self-referential identity.

\begin{lemma}\label{lem:selfref}
Assume $a\ge2$, and set $c:=b-1$. Then for every $n\ge2$,
\begin{equation}\label{eq:selfref}
  \sigma_{a\sigma_n+c} = a\sigma_n + n + c - 1.
\end{equation}
\end{lemma}

\begin{proof}
Fix $n\ge2$. By \eqref{eq:delta-gap-min}, the integer $\delta_n-1$ is not a
deletion, hence it is a survivor. Among the integers $\le \delta_n$, exactly
$n$ are deletions (namely $\delta_1,\ldots,\delta_n$), so the number of
survivors $\le \delta_n$ is $\delta_n-n$. Since $\delta_n-1$ is a survivor
and $\delta_n$ is not, $\delta_n-1$ is the $(\delta_n-n)$th survivor, i.e.
\[
  \sigma_{\delta_n-n}=\delta_n-1.
\]
Using \eqref{eq:fraenkel-oce} to rewrite $\delta_n-n=a\sigma_n+c$ and
$\delta_n-1=a\sigma_n+n+c-1$ yields \eqref{eq:selfref}.
\end{proof}

\subsection{Characteristic slope and continued fraction}

We now extract from \eqref{eq:selfref} the limiting slope of the normalized
survivors and deletions.  The key observation is that the self-referential
index sequence $m_n:=a\sigma_n+(b-1)$ has bounded gaps, so the subsequential
limits of $x_n:=\sigma_n/n$ along $(m_n)$ are the same as along the full
sequence.  This gives $L=f(L)$ for the set of limit points, and the
contraction then forces $L$ to be a singleton.

\begin{lemma}\label{lem:limit-set-same}
Assume $a\ge2$, and set
\[
  x_n:=\frac{\sigma_n}{n},
  \qquad
  m_n:=a\sigma_n+(b-1).
\]
Then the sequences $(x_n)$ and $(x_{m_n})$ have the same set of subsequential
limits.
\end{lemma}

\begin{proof}
By the two-gap property (Proposition~\ref{thm:two-gap}),
$\sigma_{n+1}-\sigma_n\in\{1,2\}$, hence
$m_{n+1}-m_n=a(\sigma_{n+1}-\sigma_n)\in\{a,2a\}$.
So the gaps of $(m_n)$ are bounded by $2a$.

Fix $M\ge1$.  For $|r|\le M$, the numerator $\sigma_{n+r}-\sigma_n=O_M(1)$
while $\sigma_n\asymp n$ by the two-gap property, so
\[
  \sup_{|r|\le M}\left|\frac{\sigma_{n+r}}{n+r}-\frac{\sigma_n}{n}\right|
  \longrightarrow 0
  \qquad(n\to\infty).
\]
Since every sufficiently large integer $N$ lies within distance $2a$ of some
$m_n$, taking $M=2a$ shows that $x_N-x_{m_n}\to0$ whenever $|N-m_n|\le2a$.
Therefore $(x_n)$ and $(x_{m_n})$ have the same set of subsequential limits.
\end{proof}

\begin{theorem}\label{thm:slope}
Assume $a\ge2$. Then the limit $\lim_{n\to\infty}\sigma_n/n$ exists and is
the unique real number $\alpha=\alpha(a)\in(1,2)$ satisfying
\begin{equation}\label{eq:alpha-eq}
  \alpha = 1 + \frac{1}{a\alpha}
  \qquad\Longleftrightarrow\qquad
  a\alpha^2-a\alpha-1=0.
\end{equation}
Equivalently,
\begin{equation}\label{eq:alpha-formula}
  \alpha(a)=\frac{1+\sqrt{1+4/a}}{2}
  =\frac{a+\sqrt{a^2+4a}}{2a}.
\end{equation}
Moreover,
\begin{equation}\label{eq:beta-slope}
  \lim_{n\to\infty}\frac{\delta_n}{n} = a\alpha(a)+1.
\end{equation}
\end{theorem}

\begin{proof}
By the two-gap property (Proposition~\ref{thm:two-gap}), $\sigma_n/n\in[1,2]$
for all $n$, so the set $L$ of subsequential limits of $x_n:=\sigma_n/n$ is a
nonempty compact subset of $[1,2]$.

Set $c:=b-1$, $m_n:=a\sigma_n+c$, and $f(x):=1+1/(ax)$.
By Lemma~\ref{lem:selfref},
\[
  \sigma_{m_n} = a\sigma_n + n + c - 1,
\]
so
\[
  x_{m_n}
  = \frac{\sigma_{m_n}}{m_n}
  = \frac{a\sigma_n+n+c-1}{a\sigma_n+c}
  = 1 + \frac{n-1}{a\sigma_n+c}.
\]
If $x_{n_k}\to x\in L$, then $x_{m_{n_k}}\to f(x)$, giving $f(L)\subseteq L$.

Conversely, every subsequential limit $y$ of $(x_{m_n})$ arises as $y=f(x)$
for some $x\in L$, so the set of subsequential limits of $(x_{m_n})$ is exactly
$f(L)$.  By Lemma~\ref{lem:limit-set-same}, $(x_n)$ and $(x_{m_n})$ share the
same subsequential limits, hence
\[
  L = f(L).
\]

Now $f$ is a strict contraction on $[1,2]$:
\[
  |f'(x)| = \frac{1}{ax^2} \le \frac{1}{a} \le \frac{1}{2}.
\]
Therefore
\[
  \operatorname{diam}(L)
  = \operatorname{diam}(f(L))
  \le \frac{1}{a}\,\operatorname{diam}(L),
\]
which forces $\operatorname{diam}(L)=0$.  Hence $L=\{\alpha\}$ is a singleton and
$\sigma_n/n\to\alpha$.  The fixed-point equation $\alpha=f(\alpha)$ is exactly
\eqref{eq:alpha-eq}, and the formula \eqref{eq:alpha-formula} follows by the
positive root of the resulting quadratic.

Finally, \eqref{eq:beta-slope} is immediate from \eqref{eq:fraenkel-oce}:
$\delta_n/n = a(\sigma_n/n)+1+(b-1)/n \to a\alpha+1$.
\end{proof}

The slope $\alpha(a)$ has an explicit periodic continued fraction.

\begin{proposition}\label{prop:cf}
Let $\alpha=\alpha(a)$ be as in Theorem~\ref{thm:slope}, and put
$\beta:=a\alpha+1$. Then
\[
  \alpha = [1;\overline{a,1}]
  \qquad\text{and}\qquad
  \beta = [a+1;\overline{1,a}].
\]
\end{proposition}

\begin{proof}
From \eqref{eq:alpha-eq} we get
\[
  \alpha = 1+\frac{1}{a\alpha}
  = 1 + \frac{1}{a + \frac{1}{\alpha}},
\]
which is the continued-fraction identity
$\alpha=[1;a,\alpha]$, hence $\alpha=[1;\overline{a,1}]$.

For $\beta=a\alpha+1$, we use $a\alpha = a + 1/\alpha$
(from \eqref{eq:alpha-eq}), giving
\[
  \beta = a + 1 + \frac{1}{\alpha} = [a+1;\overline{1,a}].
\]
\end{proof}

\begin{corollary}\label{cor:beatty-dichotomy}
The sequences $(\sigma_n)$ and $(\delta_n)$ are Beatty sequences if and only if $a=1$.
For $a=1$ ($b=0$ under the standing assumption), the rank identity $\delta_n=\sigma_n+n-1$
and $\varphi^2=\varphi+1$ force both sequences to be (non-homogeneous) Beatty sequences
(Corollary~\ref{cor:wythoff}).
For $a\ge2$, the multiplicative factor $a\ge2$ in $\delta_n=a\sigma_n+n+(b-1)$
is incompatible with the Beatty condition.  More precisely, if $\sigma_n=\lfloor\alpha n+\gamma\rfloor$
were a Beatty sequence, then $\delta_n = a\lfloor\alpha n+\gamma\rfloor+n+(b-1)$, but
$a\lfloor x\rfloor = \lfloor ax\rfloor$ fails whenever $\{x\}\ge 1/a$ (which occurs for
a positive-density set of indices since $\alpha$ is irrational), so $(\delta_n)$ cannot equal
$\lfloor\beta n+\gamma'\rfloor$ for any $\beta,\gamma'$.
This parallels Fraenkel's result~\cite{Fraenkel1998} that the $\mathcal{P}$-positions
of $\mathrm{NIM}(a,1)$ admit no Beatty representation for $a\ge2$.
\end{corollary}

\subsection{Kimberling rank transform interpretation (for \texorpdfstring{$b\ge1$}{b>=1})}

The complementary equation~\eqref{eq:fraenkel-oce} also admits a convenient formulation in terms of
Kimberling's \emph{rank transform}. We use the definition given in
\cite{Kimberling2011} (see also A187224 in~\cite{OEIS}).

\begin{definition}[Block sequence $u_{a,b}$]\label{def:uab}
Assume $a\ge1$ and $b\ge1$. Define the nondecreasing sequence
$u_{a,b}=(u_{a,b}(n))_{n\ge1}$ by
\[
  u_{a,b}(n)=\max\!\left(0,\left\lceil\frac{n-b+1}{a}\right\rceil\right).
\]
Equivalently, $u_{a,b}$ begins with $b-1$ copies of $0$, and then contains
exactly $a$ copies of each integer $m\ge1$.
\end{definition}

\begin{definition}[Rank transform $R(u)$]\label{def:ranktransform-new}
Let $u=(u_n)_{n\ge1}$ be a nondecreasing sequence of nonnegative integers, and
let $r=(r_n)_{n\ge1}$ be an increasing sequence of positive integers. Define
\[
  \tilde h(1):=u_1,\qquad
  \tilde h(n):=\#\{\,i:\ r_i\in[u_{n-1},u_n)\ \}\quad(n\ge2),
\]
and set $T_u(r)=(r'_n)$ by $r'_1:=1$ and
\[
  r'_n:=r_{n-1}+\tilde h(n)+1\qquad(n\ge2).
\]
If $r$ is a fixed point of $T_u$ (i.e.\ $T_u(r)=r$), we call $r$ the
\emph{rank transform} of $u$ and write $r=R(u)$.
(We write $\tilde h$ for Kimberling's counting function to avoid conflict
with the binary gap word $H$ of Definition~\ref{def:gapword-H} and the
sieve pointer $h_n$.)
\end{definition}

\begin{remark}
Known rank transforms include (cf.~\cite{Kimberling2007}):
$R(1,1,1,\ldots)=(1,2,3,\ldots)$;
$R(n)_{n\ge1}=\text{A000201}$ (lower Wythoff);
$R(\lfloor 3n/2\rfloor)=\text{A187224}$;
$R(2n-1)_{n\ge1}=(2n-1)_{n\ge1}$ (odd numbers form a fixed point).
\end{remark}

The sieve and the rank transform produce the same complementary pair.

\begin{theorem}\label{thm:ranktransform-new}
Assume $a\ge2$ and $b\ge1$.  The golden sieve on~$a\N+b$ and
Kimberling's rank transform~\cite{Kimberling2011} on the block
sequence $u_{a,b}$ of Definition~\textup{\ref{def:uab}} produce the
same complementary pair.  More precisely, if $R(u_{a,b})=(r_n)$ denotes
the rank transform and $c_n := n + a\,r_n + (b-1)$ its companion, then
\begin{equation}\label{eq:c-from-r}
  (\sigma_n,\,\delta_n) \;=\; (r_n,\,c_n)
  \qquad\text{and}\qquad
  \{r_n\}\sqcup\{c_n\}=\N.
\end{equation}
\end{theorem}

\begin{proof}
By Definition~\ref{def:ranktransform-new}, $r$ has a complementary sequence $c$ obtained as
the rank of $r_n$ when the multiset $\{u_1,u_2,\ldots\}$ and the set
$\{r_1,r_2,\ldots\}$ are jointly ranked (with the convention that $u$-entries
precede $r$-entries in ties).

For the specific block sequence $u_{a,b}$, the
number of $u$-entries $\le r_n$ is $(b-1)+a r_n$. Adding the $n$ entries
$r_1,\ldots,r_n$ gives that the joint-rank of $r_n$ is
\[
  c_n = n+a\,r_n+(b-1),
\]
which is \eqref{eq:c-from-r}. Hence $\{r_n\}\sqcup\{c_n\}=\N$.

On the other hand, the golden sieve produces normalized complementary sequences
$(\sigma,\delta)$ satisfying the same affine relation
$\delta_n=n+a\sigma_n+(b-1)$ (Corollary~\ref{cor:fraenkel-oce}).  Since $b\ge1$,
we have $c:=b-1\ge0$, so Lemma~\ref{lem:mex-uniqueness} applies (with $c=b-1$)
and the pair $(\sigma,\delta)$ is uniquely determined.  Hence
$(\sigma_n,\delta_n)=(r_n,c_n)$ for all $n\ge1$.
\end{proof}

\begin{remark}
The restriction $b\ge1$ in Definition~\ref{def:uab} ensures the initial block of
zeros has nonnegative length; for $b=0$ the complementary equation becomes
$\delta_n=a\sigma_n+n-1$, which requires a separate treatment.
\end{remark}
\section{Variants}\label{sec:variants}

\subsection{The golden sieve on \texorpdfstring{$\{n^2\}$}{squares}}\label{sec:squares}

We consider the golden sieve on the perfect squares
$W_0=(1^2,2^2,3^2,\ldots)$, where a new self-referential
phenomenon appears.  The hiccup rule involves a \emph{quadratic} nesting,
with gaps governed by squares of survivors.

\begin{definition}[Golden sieve on squares]\label{def:squaresieve}
Let $W_0=(1,4,9,16,25,\ldots)=(n^2)_{n\ge1}$ and run the golden sieve
(Definition~\ref{def:sieve}).  Write $s_n$ for the $n$-th survivor and $d_n$
for the $n$-th deletion (both perfect squares).  The \emph{normalized}
survivor and deletion indices are
\[
  \mu_n:=\sqrt{s_n},\qquad \lambda_n:=\sqrt{d_n},
\]
so that $s_n=\mu_n^2$ and $d_n=\lambda_n^2$.
\end{definition}
\noindent
Since $(\sigma_n),(\delta_n)$ already denote normalized sequences for
arithmetic progressions
(Sections~\ref{sec:general}--\ref{sec:hiccup-ab}) and $\nu$ denotes
the normalization map, we write $(\mu_n),(\lambda_n)$ for the
square-root indices in this subsection.

The first terms are
\begin{align*}
  (\mu_n) &= 2,3,4,5,6,7,8,9,11,12,13,14,15,16,17,19,20,21,\ldots,\\
  (\lambda_n) &= 1,10,18,28,40,54,70,88,129,153,179,207,237,269,303,\ldots
\end{align*}
Note that $\mu_n$ and $\lambda_n$ are \emph{root-indices}: the actual
survivors and deletions are the perfect squares $s_n=\mu_n^2$ and
$d_n=\lambda_n^2$ (e.g.\ the second deletion is $d_2=10^2=100$, not~$10$).

We first establish the pointer--survivor identity.

\begin{lemma}\label{lem:squares-pointer}
For $n\ge2$, the pointer at step~$n$ satisfies $h_n=s_n=\mu_n^2$.
\end{lemma}

\begin{proof}
Since $W_0(m)=m^2\ge m+1$ for all $m\ge2$, Lemma~\ref{lem:prefix-stability}
applies from step~$2$ onward.
\end{proof}

\begin{lemma}\label{lem:squares-mono}
For the golden sieve on $\{n^2\}$, $(\lambda_n)_{n\ge1}$ is strictly increasing.
\end{lemma}

\begin{proof}
By Lemma~\ref{lem:squares-pointer}, $(h_n)_{n\ge2}=(\mu_n^2)_{n\ge2}$
is strictly increasing.  At step~$n$, the value $d_n$ is deleted from
position $h_n$ of $W_{n-1}$.  Since $h_{n+1}>h_n$, the element read at
position $h_{n+1}$ in $W_n$ was at position $h_{n+1}+1$ in $W_{n-1}$,
hence
\[
  d_{n+1}=W_n(h_{n+1})=W_{n-1}(h_{n+1}+1)>W_{n-1}(h_n)=d_n.
\]
Applying $\sqrt{\cdot}$ yields $\lambda_{n+1}>\lambda_n$.
\end{proof}

This gives the rank identity for the square sieve.

\begin{proposition}\label{thm:squares-rank}
For $n=1$, the pointer and target coincide at $1^2$, so $\lambda_1=1$ and $\mu_1=2$.
For every $n\ge2$:
\begin{equation}\label{eq:squares-rank}
  \lambda_n = \mu_n^2 + (n-1).
\end{equation}
\end{proposition}

\begin{proof}
For $n=1$, $h_1=W_0(1)=1=d_1$, giving $\lambda_1=1$ (degenerate step).
For $n\ge2$, by Lemma~\ref{lem:squares-pointer} the pointer is
$h_n=\mu_n^2$.  The target $d_n=W_{n-1}(h_n)$ sits at position
$h_n=\mu_n^2$ in the working index sequence~$I_{n-1}$.  Before
step~$n$, exactly $n-1$ indices have been deleted (one per step).
By Lemma~\ref{lem:squares-mono} all prior deletions satisfy
$\lambda_k<\lambda_n$, so all $n-1$ deleted indices lie below~$\lambda_n$.
Hence
\[
  \mu_n^2 = \lambda_n-(n-1),
\]
which gives~\eqref{eq:squares-rank}.
\end{proof}

The rank identity implies a two-gap property.

\begin{proposition}\label{prop:squares-twogap}
The normalized survivor gaps satisfy
$\mu_n-\mu_{n-1}\in\{1,2\}$ for every $n\ge2$.
\end{proposition}

\begin{proof}
From the rank identity~\eqref{eq:squares-rank},
\[
  \lambda_n-\lambda_{n-1}
  = \mu_n^2-\mu_{n-1}^2+1
  = (\mu_n-\mu_{n-1})(\mu_n+\mu_{n-1})+1
  \;\ge\; \mu_n+\mu_{n-1}+1
  \;\ge\; 5
\]
for $n\ge3$ (since $\mu_n\ge 3$).  In particular, no two deletions are
consecutive.  Since $\{\mu_k\}\sqcup\{\lambda_k\}=\N$ and consecutive
deletions never occur, at most one integer can be missing between
$\mu_{n-1}$ and $\mu_n$, giving gaps in~$\{1,2\}$.
\end{proof}

We can now identify the precise membership rule.

\begin{proposition}%
\label{thm:squares-selfref}
For the golden sieve on $\{n^2\}$, for every $n\ge2$:
\begin{equation}\label{eq:squares-selfref}
  \mu_{\mu_n^2-1} = \mu_n^2+n-2.
\end{equation}
Equivalently, with $\lambda_n=\mu_n^2+(n-1)$
(Proposition~\ref{thm:squares-rank}):
$\mu_{\lambda_n-n}=\lambda_n-1$ for $n\ge2$.
\end{proposition}

\begin{proof}
The partition $\N=\{\mu_k\}\sqcup\{\lambda_k\}$ gives, for the counting
functions $S(x)=\#\{k:\mu_k\le x\}$ and $D(x)=\#\{k:\lambda_k\le x\}$,
the relation $S(x)+D(x)=x$ for all $x\ge1$.

Fix $n\ge2$ and set $x=\lambda_n-1=\mu_n^2+n-2$.  Since $(\lambda_k)$ is
strictly increasing (Lemma~\ref{lem:squares-mono}), $D(\lambda_n-1)=n-1$, hence
\[
  S(\lambda_n-1)=(\lambda_n-1)-(n-1)=\mu_n^2-1.
\]
This means $\mu_{\mu_n^2-1}\le \lambda_n-1$ and
$\mu_{\mu_n^2}\ge \lambda_n$.
Since $\lambda_n-\lambda_{n-1}\ge 5$ (as shown in
Proposition~\ref{prop:squares-twogap}), the integer $\lambda_n+1$
is not a deletion, hence $\lambda_n+1\in\{\mu_k\}$.
Together with $\mu_{\mu_n^2}\ge\lambda_n$ and
$\lambda_n\notin\{\mu_k\}$, this gives
\[
  \mu_{\mu_n^2}=\lambda_n+1.
\]
The counting identity $S(\lambda_n-1)=\mu_n^2-1$ then forces
$\mu_{\mu_n^2-1}=\lambda_n-1=\mu_n^2+n-2$.
\end{proof}

\begin{corollary}\label{cor:metahiccup}
For $n\ge2$:
\[
  \mu_n-\mu_{n-1}=2
  \quad\Longleftrightarrow\quad
  \mu_{n-1}+1\in\{\lambda_1,\lambda_2,\ldots\}.
\]
\end{corollary}

\begin{proof}
The gap $\mu_n-\mu_{n-1}=2$ precisely when the integer
$\mu_{n-1}+1$ is missing from $\{\mu_k\}$, i.e.\ belongs to
its complement $\{\lambda_k\}$.
\end{proof}

\begin{theorem}%
\label{thm:squares-hiccup-surv}
Let $(\mu_n)_{n\ge1}$ be the survivor root-indices for $W_0=\{n^2\}$, and
define the \emph{square-shadow set}
\[
  \mathcal{M} := \{\mu_m^2 : m\ge2\}\subseteq\N.
\]
Then $\mu_1=2$ and for every $n\ge2$,
\begin{equation}\label{eq:squares-hiccup-surv}
  \mu_n-\mu_{n-1}=
  \begin{cases}
    2 & \text{if } n\in \mathcal{M},\\
    1 & \text{if } n\notin \mathcal{M}.
  \end{cases}
\end{equation}
Equivalently, with $H(n):=(\mu_n-\mu_{n-1})-1\in\{0,1\}$, one has
$H(n)=\mathbf{1}_{\,n\in\mathcal{M}}$ for $n\ge2$.
\end{theorem}

\begin{proof}
By Corollary~\ref{cor:metahiccup},
$\mu_n-\mu_{n-1}=2$ iff $\mu_{n-1}+1\in\{\lambda_m\}$.

Fix $n\ge2$.  The condition $\mu_{n-1}+1=\lambda_m$ is equivalent to
$\mu_{n-1}=\lambda_m-1$.  By Proposition~\ref{thm:squares-selfref}, for every
$m\ge2$ we have $\lambda_m-1=\mu_{\mu_m^2-1}$.  Hence
\[
  \mu_{n-1}=\lambda_m-1
  \quad\Longleftrightarrow\quad
  \mu_{n-1}=\mu_{\mu_m^2-1}.
\]
Since $(\mu_n)$ is strictly increasing, equality of values forces equality
of indices: $n-1=\mu_m^2-1$, i.e.\ $n=\mu_m^2$.

It remains to check $m=1$.  Here $\lambda_1=1$ (the degenerate step), so the
condition $\mu_{n-1}+1=\lambda_1=1$ would require $\mu_{n-1}=0$, which
never occurs.  Hence $m=1$ contributes no gap-$2$ position, and
$\mathcal{M}=\{\mu_m^2:m\ge2\}$.
\end{proof}

\begin{remark}
Equation~\eqref{eq:squares-hiccup-surv} is a hiccup-type rule in which the
gap decision at index~$n$ is governed by membership of~$n$ in a set built from
the survivor sequence itself, namely its square shadow~$\mathcal{M}$.  This is
the quadratic analogue of the linear hiccup (Theorem~\ref{thm:hiccup-ab}),
with the nesting provided by the squaring map $m\mapsto\mu_m^2$ rather than
a simple linear recurrence.
\end{remark}

The nested hiccup (Theorem~\ref{thm:squares-hiccup-surv}) translates
into a counting identity that pins down the growth of~$(\mu_n)$.

\begin{lemma}\label{lem:counting-squares}
Let $S(x)=\#\{k:\mu_k\le x\}$ and $D(x)=\#\{k:\lambda_k\le x\}$ be
the survivor and deletion counting functions, so that $S(x)+D(x)=\lfloor x\rfloor$.
Then for every $n\ge2$,
\begin{equation}\label{eq:mu-counting}
  \mu_n = n + S(\sqrt{n}).
\end{equation}
\end{lemma}

\begin{proof}
By Theorem~\ref{thm:squares-hiccup-surv},
$\mu_n-\mu_{n-1}=1+\mathbf{1}_{n\in\mathcal{M}}$ for $n\ge2$,
with $\mathcal{M}=\{\mu_m^2:m\ge2\}$.  Summing from $2$ to~$n$
and using $\mu_1=2$:
\[
  \mu_n = 2+(n-1)+\#\{k\in[2,n]:k\in\mathcal{M}\}
        = n+1+\#\{m\ge2:\mu_m^2\le n\}.
\]
Now $\#\{m\ge2:\mu_m^2\le n\}=S(\sqrt{n})-1$, since $\mu_1=2$
and $\mu_m\le\sqrt{n}$ iff $\mu_m^2\le n$.
Hence $\mu_n = n+S(\sqrt{n})$.
\end{proof}

\begin{lemma}\label{lem:S-renorm-rigorous}
For $m\ge4$,
\[
  S(m) = m - S\!\bigl(\lfloor\sqrt{m}\rfloor\bigr) + O(1).
\]
\end{lemma}

\begin{proof}
Let $n:=S(m)$, so $\mu_n\le m<\mu_{n+1}$.
By Lemma~\ref{lem:counting-squares}, $\mu_k = k + S(\sqrt{k})$, hence
$n + S(\sqrt{n}) \le m < n+1+S(\sqrt{n+1})$, giving
$n = m - S(\sqrt{n}) + O(1)$.

Since $\lambda_j\ge j^2$ (as $\lambda_j=\mu_j^2+(j-1)\ge j^2$), we have
$D(m)\le\sqrt{m}+O(1)$, so $n=S(m)=m-D(m)=m+O(\sqrt{m})$ and
$\sqrt{n}=\sqrt{m}+O(1)$.  Since $S$ is nondecreasing and changes by at most~$1$
per unit, $S(\sqrt{n})=S(\lfloor\sqrt{m}\rfloor)+O(1)$, and the claim follows.
\end{proof}

\begin{theorem}\label{thm:squares-asymp}
Define the \emph{alternating root tower}
\begin{equation}\label{eq:root-tower}
  f(m) := \sum_{j=0}^{K(m)} (-1)^j\, \lfloor m^{1/2^j}\rfloor,
\end{equation}
where $K(m)=\max\{k\ge0:m^{1/2^k}\ge2\}=O(\log\log m)$.
Then
\[
  S(m) = f(m) + O(\log\log m).
\]
Consequently, for $n\ge4$,
\begin{equation}\label{eq:mu-tower}
  \mu_n = n + f(\lfloor\sqrt{n}\rfloor) + O(\log\log n).
\end{equation}
In particular,
\[
  \mu_n = n + n^{1/2} - n^{1/4} + n^{1/8} - \cdots + O(\log\log n).
\]
\end{theorem}

\begin{proof}
By Lemma~\ref{lem:S-renorm-rigorous}, $S(m) = m - S(\lfloor\sqrt{m}\rfloor)+O(1)$.
Set $t_0:=m$ and $t_{j+1}:=\lfloor\sqrt{t_j}\rfloor$ for $j\ge0$.  Iterating,
\[
  S(m) = \sum_{j=0}^{J}(-1)^j\,t_j + (-1)^{J+1}S(t_{J+1}) + O(J)
\]
for every $J\ge0$.  Choose $J=K(m)$; then $t_{K(m)+1}<2$, so
$S(t_{K(m)+1})=O(1)$.  Since each iteration contributes an error $O(1)$ and
there are $K(m)=O(\log\log m)$ iterations, the accumulated error is
$O(\log\log m)$.  Since $t_j=\lfloor m^{1/2^j}\rfloor+O(1)$, we obtain
$S(m) = f(m)+O(\log\log m)$.

The formula for $\mu_n$ follows from Lemma~\ref{lem:counting-squares}:
$\mu_n = n + S(\lfloor\sqrt{n}\rfloor) = n + f(\lfloor\sqrt{n}\rfloor) + O(\log\log n)$.
\end{proof}

\begin{openproblem}\label{op:squares}
Determine whether the binary gap word~$H$ of the golden sieve
on~$\{n^2\}$ is morphic or automatic after suitable normalization.
\end{openproblem}

\section{Extraction sieves and hiccup sequences}\label{sec:metallic}

The golden sieve is \emph{positional}: at each step it reads a position
in the working sequence and deletes the element found there, while all
other elements remain.  We now introduce a complementary family of sieves
based on a different principle: at each step, the \emph{minimum} of the
current working sequence is extracted as the next survivor, and a fixed
number of further minima are then discarded according to a self-referential
membership test.  On~$\N$, this extraction mechanism directly realizes the
hiccup recurrence.  On an arithmetic progression~$a\N+b$, the output is
again a hiccup sequence, but with parameters transformed by an explicit
affine map (Theorem~\ref{thm:extr-hiccup}, Proposition~\ref{prop:affine-action}).
In particular, the Bosma--Dekking--Steiner
sequence~\cite{BDS2018} is simply the extraction sieve
$\mathcal{C}_{1,3,2}$ on~$\N$ (Example~\ref{ex:silver-N}).

\subsection{The extraction sieve}\label{sec:extr-def}

\begin{definition}[Extraction sieve]\label{def:extraction}
Let $y,z\in\Z_{>0}$ with $y\neq z$, and $j\in\Z_{\ge0}$.
Set $m=\min(y,z)-1$ and $d=|y-z|$.
Let $W_0=(w_1<w_2<\cdots)$ be an infinite strictly increasing
sequence of positive integers.
The \emph{extraction sieve} $\mathcal{C}_{j,y,z}$ builds a sequence
$(s_n)_{n\ge1}$ as follows.
At step $n\ge1$:
\begin{enumerate}
\item[(i)] set $s_n:=\min W_{n-1}$ and remove it;
\item[(ii)] delete the next $m$ minima of $W$ (always);
\item[(iii)] let $t=n+1-j$; if $y>z$ and $t\in\{s_1,\ldots,s_n\}$,
  or $y<z$ and $t\notin\{s_1,\ldots,s_n\}$: delete $d$ additional
  minima.
\end{enumerate}
The integer~$j$ is the \emph{shift parameter}: the membership test
probes $n+1-j$, so $j=0$ tests the ``next'' index, $j=1$ tests
the current index, and $j=2$ tests the preceding index.
\end{definition}

The gap at step~$n$ is $s_{n+1}-s_n=(\text{number of elements
consumed})\times(\text{step of } W)$.  On a general working set the
analysis requires understanding how $W_n$ evolves, but when $W_0$ is
an arithmetic progression the structure simplifies completely.

The key structural observation is that the arithmetic progression structure is preserved at every step.

\begin{lemma}\label{lem:interval}
If $W_0=a\N+b=\{a+b,2a+b,3a+b,\ldots\}$ for some $a\ge1$, $0\le b<a$,
then at every step $W_{n-1}=\{s_n,\,s_n+a,\,s_n+2a,\ldots\}$:
the working sequence remains the tail of the same arithmetic
progression.
\end{lemma}

\begin{proof}
Induction on~$n$.  At step~$n$, the sieve removes $s_n$ together with
$m+(0\text{ or }d)$ consecutive successors from the front
of the progression.  The remainder is again a tail of the same progression with
common difference~$a$.
\end{proof}

\subsection{Equivalence with hiccup sequences}

\begin{theorem}\label{thm:extr-hiccup}
\leavevmode
\begin{enumerate}
\item[\textup{(a)}]
  On $W_0=\N$, the extraction sieve $\mathcal{C}_{j,y,z}$ produces
  the $(j,1,y,z)$-hiccup sequence: $s_1=1$ and
  \begin{equation}\label{eq:extr-hiccup}
    s_{n+1}-s_n=
    \begin{cases}
      y & \text{if $n+1-j\in\{s_1,\ldots,s_n\}$,}\\
      z & \text{otherwise.}
    \end{cases}
  \end{equation}
\item[\textup{(b)}]
  On $W_0=a\N+b$ with $a\ge1$, $0\le b<a$:
  \begin{equation}\label{eq:extr-affine}
    \mathcal{C}_{j,y,z}(a\N+b)\;=\;\operatorname{hiccup}(j,\;a+b,\;ay,\;az).
  \end{equation}
  That is, the survivors form the $(j,\,a+b,\,ay,\,az)$-hiccup
  sequence: $s_1=a+b$ and
  $s_{n+1}-s_n\in\{ay,\,az\}$ with the large gap $a\max(y,z)$
  occurring when the membership condition is met.
\end{enumerate}
\end{theorem}

\begin{proof}
(a) By Lemma~\ref{lem:interval} with $a=1$, $b=0$:
$W_{n-1}=\{s_n,s_n{+}1,s_n{+}2,\ldots\}$.
Step~$n$ removes $s_n$ plus $m+(0\text{ or }d)$ further elements.
If the condition in~(iii) holds, a total of $1+m+d=\max(y,z)$
elements are consumed, so $s_{n+1}=s_n+\max(y,z)$.
Otherwise, $1+m=\min(y,z)$ elements are consumed, giving
$s_{n+1}=s_n+\min(y,z)$.
Since $s_1=\min\N=1$ and the membership test is
``$n+1-j\in\{s_i\}$'', this is precisely the
$(j,1,y,z)$-hiccup recurrence.

\smallskip
(b) By Lemma~\ref{lem:interval}, $W_{n-1}$ is a tail of the
progression $a\N+b$, and each removed element advances the minimum
by~$a$.  Therefore step~$n$ advances by
$a\cdot\min(y,z)$ or $a\cdot\max(y,z)$
according to the same membership condition as in~(a).
The initial value is $s_1=\min(a\N+b)=a+b$.
The condition ``$n+1-j\in\{s_i\}$'' is unchanged (it tests step
indices against survivor \emph{values}).
Hence the survivors satisfy the $(j,\,a+b,\,ay,\,az)$-hiccup
recurrence.
\end{proof}

\begin{proposition}\label{prop:affine-action}
Applying the extraction sieve to the arithmetic progression
$a\N+b$ defines a map on the space of hiccup parameters:
\begin{equation}\label{eq:affine-action}
  T_{a,b}:\;(j,\,x,\,y,\,z)\;\longmapsto\;(j,\;a+b,\;ay,\;az).
\end{equation}
This map satisfies $T_{a',b'}\circ T_{a,b}=T_{a'a,\,a'b+b'}$,
which is the composition law of the affine group
$\mathrm{Aff}(\Z)=\{x\mapsto ax+b\}$,
with identity $T_{1,0}$.
In particular, the shift parameter~$j$ is invariant: the type
of membership test is preserved under the affine action.
\end{proposition}

\begin{proof}
The composition formula follows directly from Theorem~\ref{thm:extr-hiccup}(b)
applied twice: $\mathcal{C}_{j,y,z}(a\N+b)$ produces the
$(j,a+b,ay,az)$-hiccup, so applying a second affine step with parameters
$(a',b')$ transforms $(ay,az)$ to $(a'ay, a'az)$ and the initial value
to $a'(a+b)+b'=a'a+a'b+b'$, giving $T_{a',b'}\circ T_{a,b}=T_{a'a,a'b+b'}$.
\end{proof}

In other words, the affine change of input is entirely absorbed by the
action of $T_{a,b}$ on parameter space:
\[
  \mathcal{C}_{j,1,y,z}(a\N+b) \;=\; \mathcal{C}_{j,\,a+b,\,ay,\,az}(\N).
\]
This contrasts with the golden sieve, where the equivalence with the
Wythoff partition required the non-trivial pointer--survivor identity
(Theorem~\ref{thm:pointer-survivor}); here the affine structure is
transparent from the definition.

The affine action gives the characteristic slope immediately.

\begin{corollary}\label{cor:slope-ab}
For $y>z>0$, the extraction sieve $\mathcal{C}_{j,y,z}$ on
$W_0=a\N+b$ satisfies
\[
  \lim_{n\to\infty}\frac{s_n}{n}
  = \frac{az+\sqrt{a^2z^2+4a(y-z)}}{2}\,,
\]
the positive root of $\alpha^2-az\cdot\alpha-a(y-z)=0$.
The slope depends on $(y,z,a)$ but not on $j$ or~$b$.
The same formula holds for $z>y>0$ with $y$ and $z$ exchanged.
\end{corollary}

\begin{proof}
By Theorem~\ref{thm:extr-hiccup}(b), the survivors form the
$(j,a{+}b,ay,az)$-hiccup with gap set $\{ay,az\}$.
Let $r=\lim_{n\to\infty}s_n/n$ and let $M_n=\#\{2\le k\le n:k\in\{s_j\}\}$
be the hit count.  Summing the gap rule gives
\[
  s_n = s_1 + ay\cdot M_n + az\cdot(n-1-M_n)
      = s_1 + az(n-1) + a(y-z)M_n.
\]
The hits up to step~$n$ are the integers $k\le n$ in the image
of~$(s_j)$, so $M_n=N(n)+O(1)$ where $N(t)=\#\{k:s_k\le t\}$
is the counting function.  Setting $r=\lim s_n/n$, the duality
$\lim N(t)/t=1/r$ gives $M_n/n\to 1/r$, and dividing by~$n$
yields $r=az+a(y-z)/r$, i.e.\ $r^2-az\cdot r-a(y-z)=0$.
The positive root is as stated.  Existence of the limit follows
from the contraction argument of Theorem~\ref{thm:slope},
applied to $f(x)=az+a(y-z)/x$ on a suitable interval.
\end{proof}

\subsection{Beatty representation for $|y-z|=1$ on~$\N$}\label{sec:beatty-extr}

Although the extraction sieve on~$\N$ directly realizes the hiccup recurrence, the
\emph{Beatty representation} of the resulting hiccup, with explicit
slope and offset, is a non-trivial result.  We treat the
case $|y-z|=1$ with shift $j=1$ (the membership test probes the
current step index~$n$).  The case $j=0$ is treated
in~\cite{Cloitre2025}.

\begin{theorem}\label{thm:metallic-beatty}
Let $k\ge1$.
\begin{enumerate}
\item[\textup{(a)}] \textbf{Case $y=k+1$, $z=k$}
  \textup{(large gap on hit).}
  Let $\alpha=M_k:=(k+\sqrt{k^2+4})/2$, the positive root of
  $\alpha^2=k\alpha+1$, and define
  \begin{equation}\label{eq:metallic-beta}
    \beta_k=-\frac{(k-1)\,\alpha}{\alpha+1}.
  \end{equation}
  Then $s_n=\lfloor M_k\,n+\beta_k\rfloor$ for all $n\ge1$.

\item[\textup{(b)}] \textbf{Case $y=k$, $z=k+1$}
  \textup{(large gap on miss).}
  Let $\alpha=R_k:=\bigl((k+1)+\sqrt{(k+1)^2-4}\bigr)/2$, the
  positive root of $\alpha^2=(k+1)\alpha-1$, and define
  \begin{equation}\label{eq:reverse-beta}
    \beta_k^{\mathrm{r}}=-\frac{(k-1)\,\alpha}{\alpha-1}.
  \end{equation}
  Then $s_n=\lfloor R_k\,n+\beta_k^{\mathrm{r}}\rfloor$ for
  all $n\ge1$.
\end{enumerate}
\end{theorem}

\begin{proof}[Proof of~\textup{(a)}]
Set $b(n)=\lfloor \alpha n+\beta\rfloor$, where $\alpha=M_k$ and
$\beta=\beta_k$, and let $B=\operatorname{Im}(b)$.

We first check that $b(1)=1$.  Since
\[
\alpha+\beta
=\alpha\Bigl(1-\frac{k-1}{\alpha+1}\Bigr)
=\frac{\alpha(\alpha+2-k)}{\alpha+1},
\]
and $\alpha^2=k\alpha+1$, this becomes
\[
\frac{1+2\alpha}{\alpha+1}
=2-\frac{1}{\alpha+1},
\]
which lies in $(1,2)$.  Hence $b(1)=1$.

Now $k<\alpha<k+1$, so the gap $b(n+1)-b(n)$ is either $k$ or $k+1$.
More precisely,
\[
b(n+1)-b(n)=k+1
\iff \{\alpha n+\beta\}\ge (k+1)-\alpha
=\frac{\alpha-1}{\alpha}.
\]
On the other hand, a standard inversion argument shows that
\[
m\in B
\iff \Bigl\{\frac{m-\beta}{\alpha}\Bigr\}\ge 1-\frac1\alpha
=\frac{\alpha-1}{\alpha}.
\]
So the same threshold appears on both sides.

It remains to relate the two fractional parts.  We compute
\[
\alpha n+\beta-\frac{n-\beta}{\alpha}
=n\Bigl(\alpha-\frac1\alpha\Bigr)
+\beta\Bigl(1+\frac1\alpha\Bigr).
\]
Since $1/\alpha=\alpha-k$, we have $\alpha-1/\alpha=k$, and by the
definition of $\beta$,
\[
\beta\Bigl(1+\frac1\alpha\Bigr)
=\beta\,\frac{\alpha+1}{\alpha}
=-(k-1).
\]
Therefore
\[
\alpha n+\beta-\frac{n-\beta}{\alpha}=kn-(k-1)\in\Z.
\]
Because $\alpha$ is irrational and $\alpha+\beta\notin\Z$, the quantity
$\alpha n+\beta$ is never an integer.  Hence
\[
\{\alpha n+\beta\}
=
\Bigl\{\frac{n-\beta}{\alpha}\Bigr\}
\qquad (n\ge1).
\]

We may now combine the previous observations.  The gap $b(n+1)-b(n)$
equals $k+1$ exactly when
\[
\{\alpha n+\beta\}\ge \frac{\alpha-1}{\alpha},
\]
which is equivalent to
\[
\Bigl\{\frac{n-\beta}{\alpha}\Bigr\}\ge \frac{\alpha-1}{\alpha},
\]
and therefore to $n\in B$.  This is exactly the defining rule of the
$(1,1,k{+}1,k)$-hiccup sequence.  The two sequences have the same
initial value and the same rule for choosing the next gap, so they
coincide for all $n\ge1$.
\end{proof}

\begin{proof}[Proof of~\textup{(b)}]
The argument is the same in spirit, but the threshold condition is reversed.

Let $\alpha=R_k$ and $\beta=\beta_k^{\mathrm{r}}$.  From
\[
\alpha^2=(k+1)\alpha-1
\]
we get
\[
\alpha+\frac1\alpha=k+1
\qquad\text{and}\qquad
\frac1\alpha=(k+1)-\alpha.
\]
We then compute
\[
\alpha n+\beta+\frac{n-\beta}{\alpha}
=
n\Bigl(\alpha+\frac1\alpha\Bigr)
+\beta\Bigl(1-\frac1\alpha\Bigr).
\]
The first term is $(k+1)n$, and the second is
\[
\beta\,\frac{\alpha-1}{\alpha}=-(k-1),
\]
so the whole sum is
\[
(k+1)n-(k-1)\in\Z.
\]
Since $\alpha$ is irrational, this implies
\[
\{\alpha n+\beta\}
+
\Bigl\{\frac{n-\beta}{\alpha}\Bigr\}
=1.
\]

As in part~\textup{(a)}, the gap $b(n+1)-b(n)$ equals $k+1$ exactly when
\[
\{\alpha n+\beta\}\ge \frac{\alpha-1}{\alpha}.
\]
By the relation above, this is equivalent to
\[
\Bigl\{\frac{n-\beta}{\alpha}\Bigr\}<\frac{\alpha-1}{\alpha},
\]
hence to $n\notin B$.
So the larger gap occurs precisely at non-members, which is the rule
for the $(1,1,k,k{+}1)$-hiccup sequence.  This gives the result.
\end{proof}

The case $j=0$ follows the same method with a shifted membership test.

\begin{proposition}\label{prop:offset-j0}
Under the convention $j=0$ (the family $S(x,y,z)$
of~\textup{\cite{Cloitre2025}}),
the slopes $M_k$ and $R_k$ are unchanged, but the offsets differ.
\begin{enumerate}
\item[\textup{(a)}] \textbf{Case $y=k+1$, $z=k$,} $x=1$\textbf{.}
The $(0,1,k{+}1,k)$-hiccup satisfies
$s_n=\lfloor M_k\,n+\beta\rfloor$ for all $n\ge1$, with
\begin{equation}\label{eq:beta-j0-a}
  \beta=\frac{1-k\,\alpha}{\alpha+1},
  \qquad \alpha=M_k.
\end{equation}
\item[\textup{(b)}] \textbf{Case $y=k$, $z=k+1$,} $x=1$\textbf{.}
The $(0,1,k,k{+}1)$-hiccup satisfies
$s_n=\lfloor R_k\,n+\beta\rfloor$ for all $n\ge2$, with
\begin{equation}\label{eq:beta-j0-b}
  \beta=-(\alpha-2),
  \qquad \alpha=R_k.
\end{equation}
\end{enumerate}
The proofs follow the same fractional-part method as
Theorem~\textup{\ref{thm:metallic-beatty}}, with the hit
condition shifted from~$n\in B$ to~$n{+}1\in B$.
Concretely, for $j=0$ the gap at position~$n$ is $k+1$ iff $n\in B$
(the sequence probes $n$ directly, not $n-1$); the threshold
$\{\alpha n+\beta\}\ge(\alpha-1)/\alpha$ is therefore tested at the
same index~$n$ rather than at~$n-1$, which changes only the offset
$\beta$ while leaving the slope equation $a\alpha^2-a\alpha-1=0$
unchanged.
\end{proposition}

Together with~\cite{Cloitre2025} (which covers $j=0$ and arbitrary starting value),
these results supply Beatty representations for every $(j,x,y,z)$-hiccup
with $|y-z|=1$ and $j\in\{0,1\}$.  The following table records the named cases.

\begin{corollary}\label{cor:named-cases}
\leavevmode
\begin{center}\small
\renewcommand{\arraystretch}{1.15}
\begin{tabular}{@{}ccccll@{}}
\toprule
$j$ & $y$ & $z$ & Slope & Offset $\beta$ & Sequence \\
\midrule
$1$ & $2$ & $1$ & $\varphi\approx1.618$ & $0$
  & $\lfloor\varphi n\rfloor$ (\textsc{Wythoff}, A000201) \\
$1$ & $3$ & $2$ & $1{+}\sqrt{2}\approx2.414$ & $-\alpha/(\alpha{+}1)$
  & A086377 (\textsc{BDS})\\
$0$ & $3$ & $2$ & $1{+}\sqrt{2}$ & $(1{-}2\alpha)/(\alpha{+}1)$
  & A064437 \\
$1$ & $2$ & $3$ & $\varphi^2\approx2.618$ & $-\varphi$
  & A026352 \\
$0$ & $2$ & $3$ & $\varphi^2$ & $2{-}\varphi^2=-1/\varphi$
  & A007066 $(n\ge2)$ \\
\bottomrule
\end{tabular}
\end{center}
For $j=1$, the offsets are given by
Theorem~\ref{thm:metallic-beatty}
and for $j=0$ by
Proposition~\ref{prop:offset-j0}.
The table entries are obtained by specializing $k=1$ and $k=2$
in those results.
\end{corollary}

\subsection{Special cases and the silver sieve}\label{sec:silver}

For $j=1$, $y=2$, $z=1$, the extraction sieve $\mathcal{C}_{1,2,1}$
on~$\N$ recovers the Wythoff partition (the $(1,1,2,1)$-hiccup),
but by a different mechanism than the golden sieve:
the survivors are extracted from the bottom of the sequence rather
than read at positional indices.

As a richer illustration, we detail $j=1$, $y=3$, $z=2$:
the sieve $\mathcal{C}_{1,3,2}$, which produces the
Bosma--Dekking--Steiner (BDS) sequence~\cite{BDS2018}
on~$\N$ (as a direct dynamical realization)
and new sequences on~$a\N+b$ via the affine action.

\begin{definition}[Silver sieve]\label{def:silver}
The \emph{silver sieve} is the extraction sieve $\mathcal{C}_{1,3,2}$:
at step~$n$, extract $s_n=\min W_{n-1}$, always delete one further
minimum, and delete one additional minimum if $n\in\{s_1,\ldots,s_n\}$.
\end{definition}

\begin{example}\label{ex:silver-N}
Starting from $W_0=\N$:
\begin{center}\small
\begin{tabular}{@{}cccccl@{}}
\toprule
$n$ & $s_n$ & $n\in\{s_i\}$? & Always del. & Extra del. & $W_n$ \\
\midrule
$1$ & $1$ & yes & $\{2\}$ & $\{3\}$ & $\{4,5,6,\ldots\}$ \\
$2$ & $4$ & no  & $\{5\}$ & ---     & $\{6,7,8,\ldots\}$ \\
$3$ & $6$ & no  & $\{7\}$ & ---     & $\{8,9,10,\ldots\}$ \\
$4$ & $8$ & yes & $\{9\}$ & $\{10\}$ & $\{11,12,13,\ldots\}$ \\
\bottomrule
\end{tabular}
\end{center}
Survivors: $1,4,6,8,11,13,16,18,21,23,25,28,\ldots$\,
(gaps $3,2,2,3,2,3,2,3,2,2,3,\ldots$).
This is sequence A086377, with slope $\alpha=1+\sqrt{2}$ and
Beatty representation $s_n=\lfloor\alpha\,n+\beta\rfloor$
where $\beta=-(k{-}1)\alpha/(\alpha{+}1)=-\alpha/(\alpha{+}1)
=-1/\sqrt{2}\approx -0.707$
(Theorem~\ref{thm:metallic-beatty} with $k=2$).
\end{example}

\begin{example}\label{ex:silver-2N}
On $W_0=2\N=\{2,4,6,\ldots\}$:
by Theorem~\ref{thm:extr-hiccup}(b), the survivors form
the $(1,2,6,4)$-hiccup.  The first values are
$2,6,12,16,20,24,30,\ldots$ with gaps in $\{4,6\}$.
The slope is the positive root of $\alpha^2-4\alpha-2=0$,
namely $\alpha=2+\sqrt{6}\approx 4.449$.
\end{example}

\begin{example}\label{ex:silver-3N1}
On $W_0=3\N+1=\{4,7,10,\ldots\}$:
the survivors form the $(1,4,9,6)$-hiccup.  The first values are
$4,10,16,22,31,37,\ldots$ with gaps in $\{6,9\}$.
The slope is the positive root of $\alpha^2-6\alpha-3=0$,
namely $\alpha=3(1+\sqrt{2})\approx 7.243$.
\end{example}

\section{Conclusion and questions}\label{sec:concl}

The golden sieve turns a simple self-referential deletion rule into a
surprisingly rigid combinatorial structure.  On arithmetic progressions
$W_0=a\N+b$, it produces a complementary pair with a rank identity,
a two-gap law, and a slope governed by a quadratic equation.  In the
base case $W_0=\N$, this recovers the Wythoff pair up to shift.
For $W_0=\{n^2\}$, the same mechanism leads to a nested renormalization
and to the expansion
\[
  \mu_n=n+n^{1/2}-n^{1/4}+n^{1/8}-\cdots+O(\log\log n).
\]
The extraction sieve gives a second point of view.  It realizes hiccup
sequences dynamically and explains how passing from $\N$ to $a\N+b$
acts on the hiccup parameters: $\mathcal{C}_{j,y,z}(a\N+b)=\operatorname{hiccup}(j,a{+}b,ay,az)$.
For $|y-z|=1$, explicit Beatty formulas are available
(Theorem~\ref{thm:metallic-beatty}), covering all named cases in the table
of Corollary~\ref{cor:named-cases}.

\medskip

We close with three questions that arose naturally during this work.

\begin{enumerate}
\item \textbf{Game identification.}
  For the golden sieve on $W_0=a\N+b$, is there a
  combinatorial game whose $\mathcal{P}$-positions are encoded by the sieve partition?

\item \textbf{Combinatorics of the gap word.}
  For $a\ge2$, determine the combinatorial nature of the binary gap
  word~$H$.  Is it purely morphic?  Is it $k$-automatic for some~$k$,
  or Ostrowski-automatic with respect to the continued fraction of
  $\alpha(a)$?  What are its factor complexity and its balance properties?
  For $a=1$, the gap word is the Fibonacci word.

\item \textbf{The square sieve.}
  Determine the combinatorial class of the binary gap word for the
  golden sieve on $\{n^2\}$.  What can be said about its factor
  complexity and its balance properties?
\end{enumerate}

\subsection*{Acknowledgments}

The author thanks Robbert Fokkink and Gandhar Joshi for their
study of hiccup sequences~\cite{FJ2026}, which renewed
interest in these self-referential sequences and directly motivated
several results in this paper.  Thanks also to Neil~J.~A.~Sloane for
creating and maintaining the OEIS~\cite{OEIS}, an indispensable
resource without which the experimental side of this work would not
have been possible, and to Clark Kimberling for his work
on complementary equations and rank
transforms~\cite{Kimberling1993,Kimberling2007,Kimberling2008}.
Finally, the author is grateful to Jean-Paul Allouche and Jeffrey
Shallit for their treatise on automatic sequences and
combinatorics on words~\cite{AS2003}, which underpins much of the
word-theoretic analysis in this paper, and to Jeffrey Shallit
for popularizing \textsc{Walnut}~\cite{Shallit2022walnut}, an
automatic theorem prover for $k$-automatic, $k$-regular, and
Ostrowski-automatic sequences.
Several of the questions raised here, notably the automaticity and
factor complexity of the gap word~$H$ for $a\ge2$, appear to be
natural targets for \textsc{Walnut}.

\medskip
\noindent\textit{Generative AI disclosure.}
The AI assistant Claude (Anthropic, Claude Sonnet~4.6) was used
during the preparation of this manuscript for the following purposes:
LaTeX formatting and adaptation to journal style requirements;
assistance with numerical experimentation and verification of
sequence computations; and coding assistance for Python scripts
used to generate and check the tabulated data.
All mathematical content, results, and proofs are the sole work
of the author, who has reviewed all AI-assisted outputs and takes
full responsibility for the accuracy and integrity of the
submitted work.

\appendix

\section{OEIS catalogue}\label{sec:oeis}

We collect here the OEIS identifications explicitly used in the paper.

\begin{table}[ht]
\centering
\begin{tabular}{@{}lllp{0.38\textwidth}@{}}
\toprule
Object & Notation & OEIS & First terms / description \\ \midrule
Sieve on $\N$ (survivors) & $(s_n)$ & A099267 &
  $2,3,5,6,8,10,11,13,14,16,\ldots$ \\
Sieve on $\N$ (deletions) & $(d_n)$ & A007066 &
  $1,4,7,9,12,15,17,20,22,25,\ldots$ \\
Lower Wythoff sequence & $(A_n)$ & A000201 &
  $A_n=s_{n+1}-2=\lfloor n\varphi\rfloor$ \\
Upper Wythoff sequence & $(B_n)$ & A001950 &
  $B_n=d_{n+1}-2=\lfloor n\varphi^2\rfloor$ \\
Gap word ($a=1$) & $H$ & A003849 &
  Fibonacci word: $0,1,0,0,1,0,1,0,0,1,\ldots$ \\
\bottomrule
\end{tabular}
\caption{Core OEIS sequences cited in the paper.}
\label{tab:oeis-core}
\end{table}

\medskip
\noindent\textbf{Fraenkel sequences ($b=1$):}
\begin{itemize}
\item $a=1$: lower Wythoff A000201, upper Wythoff A001950.
\item $a=2$: A003622 (survivors), A003623 (deletions).
\end{itemize}

\medskip
\noindent\textbf{Hiccup sequences:}
\begin{itemize}
\item $(0,1,2,3)$-hiccup: A007066 (= deleted values for $W_0=\N$).
\item $(1,1,2,1)$-hiccup: A000201 (= lower Wythoff = survivors,
  shifted).
\item $(1,1,3,2)$-hiccup: A086377 (= silver sieve on~$\N$ =
  Bosma--Dekking--Steiner sequence, Corollary~\ref{cor:named-cases}).
\item General $(j,x,y,z)$-hiccup: see the table in \cite{FJ2026}.
\end{itemize}

\medskip
\noindent\textbf{Rank transforms (Kimberling):}
\begin{itemize}
\item $R\bigl((n)\bigr) = \text{A000201}$ (lower Wythoff), with
  complement A001950.
\item $R\bigl(\lfloor 3n/2\rfloor\bigr) = \text{A187224}$, with
  complement A187225.  The input $\lfloor 3n/2\rfloor = \text{A032766}$.
\item See Definition~\ref{def:ranktransform-new} and
  Theorem~\ref{thm:ranktransform-new} for the connection to sieve survivor recurrences.
\end{itemize}

\noindent\textbf{Disclosure statement.} No competing interests are reported by the author.

\end{document}